%% file: semilinear_elliptic_systems_nonlocal_homogeneous_critical_nonlinearities_arxiv.tex
\documentclass[reqno, 11pt]{amsart}
\input{structure.tex}

\renewcommand{\epsilon}{\varepsilon}

\usepackage[pagewise]{lineno}
\usepackage[normalem]{ulem}

\title[Nonlocal homogeneous nonlinearities]{Semilinear Elliptic Systems with Nonlocal Homogeneous Critical Nonlinearities} 
\author{Colin Fried and Mathew Gluck}
\address{Southern Illinois University \\
	School of Mathematical and Statistical Sciences}
\email{mathew.gluck@siu.edu}
\date{\today}
\thanks{This material is based upon work supported by the National Science Foundation under Grant No. DMS-2418889.}
\newif\ifdetails
\colorlet{details}{VioletRed}
\begin{document} 
\begin{abstract}
This paper concerns a variational system of nonlinear elliptic equations that generalizes the classical Brezis-Nirenberg problem. In addition to considering vector-valued unknown functions in place of scalar-valued unknown functions, the problem under consideration generalizes the Brezis-Nirenberg problem in two directions. First, the nonlinearity is the sum of two homogeneous functions of suitable homogeneity degrees. Second, the Sobolev-critical term in the nonlinearity is nonlocal. We establish conditions under which a nontrivial solution to the system under consideration is guaranteed to exist and conditions under which the system under consideration admits only the trivial solution. 
\end{abstract}
\flushbottom
\maketitle
\thispagestyle{empty}
%

\section{Introduction}
In 1983 Brezis and Nirenberg \cite{BrezisNirenberg1983} determined conditions on $\lambda\in \bb R$ and the bounded domain $\Omega \subset \bb R^n$ ($n\geq 3$) for which the problem 
\begin{equation}
\label{eq:BN_problem}
\begin{cases}
	-\lap u = \lambda u + |u|^{\frac{4}{n -2}}u & \text{ in }\Omega\\
	u= 0 & \text{ on }\bdy \Omega	
\end{cases}
\end{equation}
admits a positive solution and conditions on these objects under which problem \eqref{eq:BN_problem} does not admit a positive solution. They established the following theorem. In the statement of the theorem, $\lambda_1 = \lambda_1(-\lap)$ is the first eigenvalue of the Dirichlet Laplacian. 
\begin{oldtheorem}
\label{theorem:BN}
Let $n\geq 3$ and let $\Omega \subset \bb R^n$ be a bounded open set. 
\begin{enumerate}[label = \bf{(\alph*)}]
	\item If $n = 3$ then there are constants $\lambda_*(\Omega)\leq \lambda^*(\Omega)$ satisfying $0< \lambda_*\leq \lambda^*<\lambda_1$ such that problem \eqref{eq:BN_problem} admits a positive solution if $\lambda\in (\lambda^*, \lambda_1)$ and problem \eqref{eq:BN_problem} does not admit a positive solution if $\lambda\in (0, \lambda_*]$.
	\item If $n\geq 4$ then problem \eqref{eq:BN_problem} admits a positive solution if and only if $\lambda\in (0, \lambda_1)$. 
\end{enumerate}
\end{oldtheorem}
The subtleties under which problem \eqref{eq:BN_problem} is solvable, at least in the case that one in interested in positive solutions, are apparent in the statement of Theorem \ref{theorem:BN}. Due in part to these subtleties, problem \eqref{eq:BN_problem} has drawn considerable attention and there is a substantial body of mathematical research devoted to extensions, generalizations and further investigations. With regard to exploring extensions of problem \eqref{eq:BN_problem}, one theme present in the literature is concerned with extending the analysis of problem \eqref{eq:BN_problem} to the setting of vector-valued unknown functions \cite{AlvesEtAl2000, Amster2002, BartschGuo2006, BrownEtAl2023}. In the vector-valued setting, a number of works have investigated the case where the analog of the right-hand side of the differential equation in problem \eqref{eq:BN_problem} is replaced by a sum of homogeneous functions \cite{Morais1999,BarbosaMontenegro2011,LuShen2020, Gluck2025anisotropic}. Such investigations are more interesting in the vector-valued setting than in the scalar-valued setting due to the fact that in the vector-valued setting there are many more homogeneous functions than there are in the scalar-valued setting.

A second theme in the extension of problem \eqref{eq:BN_problem} is analyzing terms with critical nonlinearity of Choquard type. For example, in \cite{GaoYang2018} the problem
\begin{equation}
\label{eq:BN_choquard}
\begin{cases}
	-\lap u = \lambda u + I_\mu[|u|^{2^*_\mu}]|u|^{2^*_\mu - 2} u & \text{ in }\Omega\\
	u= 0 & \text{ on }\bdy \Omega	
\end{cases}
\end{equation}
was considered, where $\mu\in (0, n)$, $I_\mu$ is the convolution operator
\begin{equation}
\label{eq:riesz_potential}
	I_\mu[\varphi](x) = \int_{\bb R^n}\frac{\varphi(y)}{|x- y|^\mu}\; \d y,
\end{equation}
and 
\begin{equation}
\label{eq:2star_mu}
	2^*_\mu = \frac{2n - \mu}{n - 2}
\end{equation}
is the so-called upper critical exponent. In addition to obtaining existence and nonexistence results analogous to those obtained in \cite{BrezisNirenberg1983}, the authors in \cite{GaoYang2018} establish existence of nontrivial solutions to \eqref{eq:BN_choquard} whenever $\lambda>0$ is not an eigenvalue of the Dirichlet Laplacian.  

In this work we consider a vector-valued generalization of problem \eqref{eq:BN_choquard} where the right-hand side of the differential equation is replaced by the sum of two vector-valued homogeneous functions, one is homogeneous of degree $1$ and the other is both nonlocal and homogeneous of degree $2\cdot 2^*_\mu - 1$. These functions generalize the functions $u\mapsto \lambda u$ and $u\mapsto I_\mu[|u|^{2^*_\mu}]|u|^{2^*_\mu - 2}u$ present in problem \eqref{eq:BN_choquard}. To describe the problem of interest more specifically, we introduce the following notions of homogeneity. 
\begin{defn}
\label{defn:homogeneous_function}
Let $q\in \bb R$ and let $H:\bb R^d\to \bb R$. We say $H$ is \emph{homogeneous of degree $q$} if $H(\varrho s) = \varrho^qH(s)$ for all $(\varrho, s)\in (0, \infty)\times \bb R^d$.  We say $H$ is \emph{positively homogeneous of degree $q$} if $H$ is homogeneous of degree $q$ and $H(s)> 0$ for all $s\in \bb R^d\setminus\{0\}$. 
\end{defn}
Examples of homogeneous functions of degree $q$ on $\bb R^d$ include $H(s) = \sum_{j = 1}^dc_j|s_j|^q$ for $c_j\in \bb R$, $H(s) = |\lb As, s\rb|^{(q- 2)/2}\lb As, s\rb$ where $A$ is a $d\times d$ matrix and $\lb \cdot, \cdot \rb$ is the Euclidean inner product, $H(s) = \prod_{j = 1}^d|s_j|^{\alpha_j}$ where $\sum_{j = 1}^d \alpha_j = q$ and $H(s) = \pi_\ell(|s_1|, \ldots, |s_d|)^{q/\ell}$, where $\pi_\ell$ is the $\ell^{\text{th}}$ elementary symmetric polynomial in $d$ variables.

The generalization of problem \eqref{eq:BN_choquard} with which we are concerned is
\begin{equation}
\label{eq:main_problem}
\begin{cases}
	-\lap U = f(U) + I_\mu[H(U)]h(U) & \text{ in }\Omega\\
	U = 0 & \text{ on }\bdy \Omega, 
\end{cases}
\end{equation}
where $U = (u_1, \ldots, u_d):\Omega \to \bb R^d$ is a vector-valued function and $\lap$ acts on $U$ in the coordinatewise sense: 
\begin{equation*}
	\lap U 
	= (\lap u_1, \ldots, \lap u_d). 
\end{equation*}
Problem \eqref{eq:main_problem} is assumed to be in potential form in the sense that the nonlinearities $f$ and $h$ satisfy
\begin{equation}
\label{eq:the_gradients}
	f = \frac12\Grad F
	\qquad \text{ and }\qquad 
	h=\frac 1{2_\mu^*} \Grad H
\end{equation} 
for some $2$-homogeneous function $F\in C^1(\bb R^d; \bb R)$ and some $2_\mu^*$-homogeneous function $H\in C^1(\bb R^d; \bb R)$. To make sense of the nonlocal term in \eqref{eq:main_problem} we assume that the definition of $H\circ U$ is extended by zero to $\bb R^n\setminus\Omega$. We observe that by choosing 
\begin{equation}
\label{eq:dFH}
	d = 1, 
	\quad 
	F(s) = \lambda s^2, 
	\quad \text{ and }\quad
	H(s) = |s|^{2^*_\mu}
\end{equation}
in problem \eqref{eq:main_problem} one recovers problem \eqref{eq:BN_choquard} so that problem \eqref{eq:main_problem} is indeed a generalization of problem \eqref{eq:BN_choquard}.

Similarly to the classical Brezis-Nirenberg problem \cite{BrezisNirenberg1983}, our objective is to establish conditions under which nontrivial solutions to problem \eqref{eq:main_problem} are guaranteed to exist and conditions under which problem \eqref{eq:main_problem} admits only the trivial solution. However, in contrast to the the classical Brezis-Nirenberg problem, where existence of a positive solution can routinely be deduced from the existence of a nontrivial solution, in general it is unclear whether the existence of a nontrivial solution to problem \eqref{eq:main_problem} implies the existence of a positive solution (a solution whose coordinate functions are nonnegative and nonzero). We refer to  \cite{Morais1999, LuShen2020, BrownEtAl2023} for some restrictions on $F$ and $H$ under which such an implication holds. Because we aim to establish results with general homogeneous nonlinearities we consider only the existence and nonexistence of nontrivial solutions to problem \eqref{eq:main_problem}, without insisting on any positivity conditions. 

To proceed further we introduce the  following smoothness and homogeneity assumptions that will be imposed on $F$ and $H$: 
\begin{enumerate}[label = {\bf A\arabic*.}, ref = {\bf A\arabic*}]
	\item $F\in C^1(\bb R^d)$ is homogeneous of degree $2$ \label{item:F_homogeneous}
	\item $H\in C^1(\bb R^d)$ is positively homogeneous of degree $2_\mu^*$. \label{item:H_positively_homogeneous}
\end{enumerate}
Following the strategy used by Brezis and Nirenberg in \cite{BrezisNirenberg1983} and by \cite{Yamabe1960, Trudinger1968, Aubin1976-2, Schoen1984} in the resolution of the Yamabe problem, we will establish the existence of nontrivial solutions to problem \eqref{eq:main_problem} by showing that a suitable energy functional subject to a suitable constraint has sufficiently small infimum. In this context, it is natural to impose both a positivity assumption on $H$ and assumptions on $F$ under which the energy functional is coercive. We collect and label these assumptions in \eqref{assumptions:FH} below as they will be used repeatedly in the sequel. To state the assumptions we define
\begin{equation}
\label{eq:minF_psphere}
	M_F = \max_{s\in \bb S^{d - 1}}F(s)
	\qquad 
	\text{ and }
	\qquad
	m_F= 	\min_{s\in \bb S^{d - 1}}F(s), 
\end{equation}
and we let $\lambda_1$ be the first eigenvalue of the Dirichlet Laplacian on $\Omega$. 
The assumptions to be used in our existence theorem are 
\begin{equation}
\label{assumptions:FH}\tag{FH}
\begin{cases}
	F \text{ satisfies both \ref{item:F_homogeneous} and }M_F\in (0, \lambda_1), \text{ and }\\
	H \text{ satisfies \ref{item:H_positively_homogeneous}}. 
\end{cases}
\end{equation}

With these conventions in place we can state the main results of the present work. Our first result is an existence theorem. 
\begin{theorem}
\label{theorem:existence}
Let $n\geq 4$, let $\Omega\subset \bb R^n$ be a bounded open set, and let $\mu\in (0,n)$. Suppose $F$ and $H$ satisfy assumptions \eqref{assumptions:FH} and that $f$ and $h$ are as in \eqref{eq:the_gradients}. If there is $\theta\in \bb S^{d -1}$ for which both $H(\theta) = M_H$ and $F(\theta)>0$ then problem \eqref{eq:main_problem} admits a nontrivial weak solution. 
\end{theorem}
As mentioned above, problem \eqref{eq:main_problem} is a generalization of problem \eqref{eq:BN_choquard}. Naturally, the existence result in Theorem \ref{theorem:existence} generalizes the analogous existence result in \cite{GaoYang2018} (see Theorem 1.4 (i) in the case $\lambda\in (0,\lambda_1)$ in \cite{GaoYang2018}). In fact, our existence result goes beyond simply extending the existence result of \cite{GaoYang2018} to the vector-valued setting. Indeed, in Theorem \ref{theorem:existence}, we allow the possibility that $F$ be negative somewhere as long as it is positive at some maximizer of $H|_{\bb S^{d - 1}}$. Our next result is a nonexistence result. 
\begin{theorem}
\label{theorem:nonexistence}
Let $n\geq 3$, let $\Omega\subset \bb R^n$ be a bounded open set with $\bdy\Omega \in C^{1, 1}$, and let $\mu\in (0, n)$. Suppose $F$ and $H$ satisfy \ref{item:F_homogeneous} and \ref{item:H_positively_homogeneous} respectively and that $f$ and $h$ are as in \eqref{eq:the_gradients}.  If $\Omega$ is star-shaped and $M_F< 0$ then problem \eqref{eq:main_problem} admits no nontrivial weak solutions.
\end{theorem}
Theorem \ref{theorem:nonexistence} will be proven with the aid of an identity of Pohozaev type established in \cite{PucciSerrin1986, Degiovanni2003} that holds for weak solutions to \eqref{eq:main_problem} possessing $C^1$ regularity up to the boundary of $\Omega$. The $C^1$-regularity of weak solutions to problem \eqref{eq:main_problem} will be established by extending an argument of Moroz and Van Schaftingen \cite{MorozVanSchaftingen2015} to the vector-valued setting. 

This paper is organized as follows. In Section \ref{s:preliminaries} we discuss some preliminary notions including notational conventions, regularity of solutions to problem \eqref{eq:main_problem}, the functional space to be used in the variational argument, and a vector-valued eigenproblem involving $\Grad F$. In Section \ref{s:sobolev_inequalities} we derive a sharp Sobolev-type inequality. The sharp constant in this inequality will be used in Section \ref{s:sufficient_condition_for_minimizer} to express a sufficient condition for the existence of a nontrivial solution to problem \eqref{eq:main_problem}. This condition is expressed as a smallness threshold on the infimum of a suitably constrained energy functional. In Section \ref{s:existence} we prove Theorem \ref{theorem:existence} by verifying that this threshold is met. In Section \ref{s:non_existence} we prove Theorem \ref{theorem:nonexistence} by using a variant of the well-known Pohozaev identity. Section \ref{s:appendix} is an appendix where we provide details of some computations that may be useful to some readers but whose inclusion in the main body of the manuscript would detract from the presentation.
\section{Preliminaries}
\label{s:preliminaries}
The Hardy-Littlewood-Sobolev inequality will play a central role in our analysis. For brevity we refer to this inequality as the \emph{HLS inequality}. 
\begin{oldtheorem}[HLS inequality]
\label{theorem:hls_inequality}
For every $n\geq 1$ and every $p, q\in (1, \infty)$ satisfying $\frac 1q = \frac 1p - \frac{n - \mu}n$ there is an optimal constant $\mc H(n, \mu, p)>0$ such that the inequality 
\begin{equation*}
	\|I_\mu[f]\|_{L^q(\bb R^n)}\leq \mc H(n, \mu, p)\|f\|_{L^p(\bb R^n)} 
\end{equation*}
holds for all $f\in L^p(\bb R^n)$. 
\end{oldtheorem}
The following lemma lists some basic properties of homogeneous functions that will be used throughout the manuscript. Since the proof is straight-forward we omit the details.
\begin{lemma}[Properties of homogeneous functions]
\label{lemma:properties_homogeneous_functions}
If $H:\bb R^d\to \bb R$ is homogeneous of degree $q>0$ then the following properties hold: 
\begin{enumerate}[label = {\bf (\alph*)}, ref = {\bf(\alph*)}]
	\item \label{item:homogeneous_zero_zero} $H(0) = 0$. 
	\item If $H\in C^1(\bb R^d)$ then for any $j\in\{1, \ldots, d\}$, the partial derivative $\partial_jH\in C^0(\bb R^d)$ is homogeneous of degree $q- 1$. \label{item:derivatives_homogeneous}
	\item If $H\in C^1(\bb R^d)$ then for every $s\in \bb R^d$ we have $\lb s, \Grad H(s)\rb = qH(s)$. \label{item:homogeneous_grad_dot}
	\item \label{item:homogeneous_upper_bound} If $H\in C^0(\bb R^d)$ then there is $s\in \bb S^{d - 1}$ such that $H(s) = M_H= \max\{H(t): t\in \bb S^{d -1}\}$. Moreover, the inequality $H(s) \leq M_H|s|^q$ holds for all $s\in \bb R^d$. 
\end{enumerate}
\end{lemma}
The following lemma relates the integrability of a vector-valued function $U$ to that of the composition of $U$ with a real-valued homogeneous function. Since its proof is a simple consequence of Lemma \ref{lemma:properties_homogeneous_functions}\ref{item:homogeneous_upper_bound}, the details are omitted.
\begin{lemma}
\label{lemma:homogeneous_integrability}
Let $\Omega\subset \bb R^n$ be a measurable set (not necessarily bounded). If $H\in C^0(\bb R^d)$ is homogeneous of degree $q>0$ and if $U\in L^p(\Omega; \bb R^d)$ for some $p\geq 1$ then $H\circ U\in L^{p/q}(\Omega)$ with $\|H(U)\|_{L^{p/q}(\Omega)}\leq M_{|H|}\|U\|_{L^p(\Omega; \bb R^d)}^q$. 
\end{lemma}
The following proposition states the regularity properties of weak solutions to \eqref{eq:main_problem} under a standard smoothness assumption on $\bdy \Omega$. A proof is provided in Subsection \ref{ss:regularity}. 
\begin{prop}
\label{prop:C1alpha_regularity}
Let $n\geq 3$, let $\Omega \subset \bb R^n$ be a bounded open set with $\bdy \Omega\in C^{1, 1}$, and let $\mu\in (0, n)$. Suppose $F$ and $H$ satisfy \ref{item:F_homogeneous} and \ref{item:H_positively_homogeneous} respectively and that $f$ and $h$ are as in \eqref{eq:the_gradients}. If $U$ is a weak solution to \eqref{eq:main_problem} then there is $\alpha\in (0, 1)$ for which $U\in C^{1, \alpha}(\bar \Omega; \bb R^d)$. 
\end{prop}
\subsection{The Variational Problem}
\label{ss:variational_problem}
In this subsection we introduce a constrained minimization problem whose solutions correspond to nontrivial weak solutions to problem \eqref{eq:main_problem}. We assume for the remainder of this section that $\Omega\subset\bb R^n$ is a bounded open set. By performing elementary computations involving Lemma \ref{lemma:homogeneous_integrability}, one can verify that if $F\in C^0(\bb R^d)$ is homogeneous of degree $2$ then the functional $\Phi_{F}:H_0^1(\Omega; \bb R^d)\to \bb R$ given by 
\begin{equation}
\label{eq:intro_Phi_AF}
	\Phi_{F}(U)
	= \|U\|_{H_0^1(\Omega; \bb R^d)}^2  - \int_\Omega F(U)
\end{equation}
is well-defined. Similarly, Lemma \ref{lemma:homogeneous_integrability}, H\"older's inequality, and the HLS inequality guarantee that if $H\in C^0(\bb R^d)$ is homogeneous of degree $2^*_\mu$ then the functional $\Psi:L^{2^*}(\Omega; \bb R^d)\to \bb R$ given by 
\begin{equation}
\label{eq:intro_Psi}
	\Psi(U) = \Psi_{\mu, H}(U) = \int_\Omega I_\mu[H(U)]H(U)
\end{equation}
is well-defined. If in addition, $H$ is positively $2^*_\mu$-homogeneous then using the fact that $\Psi(U)>0$ whenever $U\in L^{2^*}(\Omega; \bb R^d)\setminus\{0\}$, it is routine to verify that the constraint manifold $\mc M$ defined by 
\begin{equation}
\label{eq:the_constraint}
	\mc M 
	= \{U\in H_0^1(\Omega; \bb R^d): \Psi(U)= 1\}  
\end{equation} 
is nonempty. 
If $F$ and $H$ satisfy \ref{item:F_homogeneous} and \ref{item:H_positively_homogeneous} respectively then up to a positive constant multiple, every critical point of $\Phi_{F}$ subject to the constraint $\mc M$ is a weak solution to problem \eqref{eq:main_problem}. 
Moreover, if $\Phi_{F}$ is coercive then the restriction $\Phi_{F}\big|_{\mc M}$ is bounded below so we may seek minimizing critical points.

In what follows we formulate, under suitable  hypotheses on $F$, an equivalent condition for the coercivity of $\Phi_{F}$. A simple construction shows that if $F\in C^0(\bb R^d)$ is a degree-$2$ homogeneous function for which $M_F>0$, then there is $U\in H_0^1(\Omega; \bb R^d)$ for which $\int_\Omega F(U)> 0$, see Lemma \ref{lemma:simple_construction} in the appendix for details. For any such $F$, and for any $U\in H_0^1(\Omega; \bb R^d)$ for which $\int_\Omega F(U)> 0$, Lemma \ref{lemma:properties_homogeneous_functions} \ref{item:homogeneous_upper_bound} and the variational characterization of the first Dirichlet eigenvalue $\lambda_1$ for $-\lap$ on $\Omega$ gives
\begin{equation*}
	\int_\Omega F(U)
	\leq M_F\int_\Omega|U|^2
	\leq \lambda_1^{-1}M_F\|U\|_{H_0^1(\Omega; \bb R^d)}^2.  
\end{equation*}
%
In view of this inequality, the quantity
\begin{equation}
\label{eq:lambda1_FA}
\begin{split}
	\lambda_{1, F}
	& = \inf\left\{\frac{\|U\|_{H_0^1(\Omega; \bb R^d)}^2}{\int_\Omega F(u)}: U\in H_0^1(\Omega ;\bb R^d) \text{ and }\int_\Omega F(U)> 0\right\}\\
	& = \inf\{\|U\|_{H_0^1(\Omega; \bb R^d)}^2: U\in H_0^1(\Omega; \bb R^d) \text{ and }\int_\Omega F(U) = 1\}
\end{split}
\end{equation}
is well-defined and positive. The following lemma characterizes $\lambda_{1, F}$ as the smallest eigenvalue for a suitable eigenproblem. Since the proof follows from a routine argument involving Rellich's Theorem, we omit the details. 
\begin{lemma}
Let $\Omega \subset \bb R^n$ be a bounded open set. If $F$ satisfies both \ref{item:F_homogeneous} and $M_F>0$, then  $\lambda_{1, F}$ as defined in \eqref{eq:lambda1_FA} is the smallest positive eigenvalue for the problem 
\begin{equation}
\label{eq:LAp_f_eigenproblem}
\begin{cases}
	-\lap U = \lambda f(U) & \text{ in }\Omega\\
	U = 0 & \text{ on }\bdy\Omega, 
\end{cases}
\end{equation}
where $f = \frac{\Grad F}{2}$. 
\end{lemma}
The following lemma relates the values of $\lambda_1$ and $\lambda_{1, F}$. 
\begin{lemma}
\label{lemma:relate_lambda1_to_lambda1F}
Let $\Omega\subset \bb R^n$ be a bounded open set. If $F\in C^0(\bb R^d)$ is homogeneous of degree two and satisfies $M_F>0$ then the first eigenvalue $\lambda_{1, F}$ defined in \eqref{eq:lambda1_FA} satisfies
\begin{equation*}
	M_F\lambda_{1, F} = \lambda_1,  
\end{equation*}
where $\lambda_1 = \lambda_1(-\lap)>0$ is the first eigenvalue of the Dirichlet Laplacian on $\Omega$. 
\end{lemma}
\begin{proof}
For any $U\in H_0^1(\Omega; \bb R^d)$ for which $\int_\Omega F(U) = 1$, Lemma \ref{lemma:properties_homogeneous_functions} \ref{item:homogeneous_upper_bound} gives $1\leq M_F\|U\|_2^2$, from which we obtain 
\begin{equation*}
	\|U\|_{H_0^1(\Omega; \bb R^d)}^2
	\geq \frac{\|U\|_{H_0^1(\Omega; \bb R^d)}^2}{M_F\|U\|_2^2}
	\geq M_F^{-1}\lambda_1. 
\end{equation*}
Taking the infimum over all $U\in H_0^1(\Omega; \bb R^d)$ satisfying $\int_\Omega F(U) = 1$ gives $M_F\lambda_{1, F}\geq \lambda_1$. To show the reverse inequality, let $s = (s_1, \ldots, s_d)\in \bb S^{d - 1}$ satisfy $M_F = F(s)$ and let $\varphi\in H_0^1(\Omega)$ be a positive first eigenfunction for $\lambda_1$. Setting $U = (s_1\varphi, \ldots, s_d\varphi)$ we have $\|U\|_{H_0^1(\Omega; \bb R^d)} = \|\varphi\|_{H_0^1(\Omega)}$
and 
\begin{equation*}
	\int_\Omega F(U)
	= M_F \int_\Omega |\varphi|^2
	> 0.
\end{equation*}
Therefore, 
\begin{equation*}
	\lambda_{1, F}
	\leq \frac{\|U\|_{H_0^1(\Omega; \bb R^d)}^2}{\int_\Omega F(U)}
	= \frac{\|\varphi\|_{H_0^1(\Omega)}^2}{M_F\|\varphi\|_2^2}
	= M_F^{-1}\lambda_1. 
\end{equation*}
\end{proof}
To close this section we establish the following equivalent condition for coercivity of $\Phi_{F}$. 
\begin{lemma}
\label{lemma:Phi_AF_coercive}
Let $n\geq 3$ and let $\Omega \subset \bb R^n$ be a bounded open set. If  $F\in C^0(\bb R^d)$ is homogeneous of degree $2$ and if $M_F> 0$ then the functional $\Phi_{F}$ defined in \eqref{eq:intro_Phi_AF} is coercive if and only if $M_F< \lambda_1$. 
\end{lemma}
\begin{proof}
If $\Phi_{F}$ is coercive then there is $\epsilon\in (0, 1)$ such that the inequality 
\begin{equation*}
	\|U\|_{H_0^1(\Omega; \bb R^d)}^2 - \int_\Omega F(U)
	\geq \epsilon\|U\|_{H_0^1(\Omega; \bb R^d)}^2
\end{equation*}
holds for all $U\in H_0^1(\Omega; \bb R^d)$. In particular, for any such $\epsilon$ and for any $U\in H_0^1(\Omega; \bb R^d)$ for which $\int_\Omega F(U)> 0$ we have 
\begin{equation*}
	1\leq (1 - \epsilon)\frac{\|U\|_{H_0^1(\Omega; \bb R^d)}^2}{\int_\Omega F(U)}. 
\end{equation*}
Taking the infimum over all such $U$ and in view of the positivity of $\lambda_{1, F}$ and Lemma \ref{lemma:relate_lambda1_to_lambda1F} we have 
\begin{equation*}
	1
	\leq ( 1- \epsilon)\lambda_{1, F}
	< \lambda_{1, F}
	= M_F^{-1}\lambda_1. 
\end{equation*}
This shows that the inequality $M_F< \lambda_1$ is a necessary condition for the coercivity of $\Phi_{F}$. Next we proceed to show that this inequality is also sufficient for the coercivity of $\Phi_{F}$. For all $U\in H_0^1(\Omega; \bb R^d)$ for which $\int_\Omega F(U)\leq 0$, the inequality $\Phi_{F}(U)\geq \|U\|_{H_0^1(\Omega; \bb R^d)}^2$ holds by inspection of $\Phi_{F}$. For $U\in H_0^1(\Omega; \bb R^d)\setminus\{0\}$ satisfying $\int_\Omega F(U)> 0$ we have
\begin{equation*}
\begin{split}
	\Phi_{F}(U)
	& = \left( 1- \frac{\int_\Omega F(U)}{\|U\|_{H_0^1(\Omega; \bb R^d)}^2}\right)\|U\|_{H_0^1(\Omega; \bb R^d)}^2\\
	& \geq \left(1 - \frac{1}{\lambda_{1, F}}\right)\|U\|_{H_0^1(\Omega; \bb R^d)}^2\\
	& = \left(1 - \frac{M_F}{\lambda_1}\right)\|U\|_{H_0^1(\Omega; \bb R^d)}^2,  
\end{split}
\end{equation*}
where the final equality holds by Lemma \ref{lemma:relate_lambda1_to_lambda1F}. We conclude that the inequality $M_F< \lambda_1$ guarantees the coercivity of $\Phi_{F}$. 
\end{proof}
%
\section{Inequalities of Sobolev Type}
\label{s:sobolev_inequalities}
In Section \ref{s:sufficient_condition_for_minimizer}, we will provide a sufficient condition for the existence of nontrivial solutions to problem \eqref{eq:main_problem}. This condition will be expressed as a smallness threshold for the infimum of the energy functional $\Phi_{F}$ constrained to $\mc M$ and this threshold is quantified in terms of a sharp Sobolev type constant. The purpose of this section is to introduce and compute the value of this sharp constant.

We denote the sharp constant in the Sobolev inequality by 
\begin{equation}
\label{eq:sharp_sobolev_constant}
	\mc S^{-1} 
	= \inf\left\{\|u\|_{D^{1, 2}(\bb R^n)}^2: u\in D^{1, 2}(\bb R^n)\text{ and }\|u\|_{L^{2^*}(\bb R^n)} = 1\right\},  
\end{equation}
where $D^{1, 2}(\bb R^n)$ is the completion of $C_c^\infty(\bb R^n)$ relative to the norm $\|u\|_{D^{1, 2}(\bb R^n)} = \|\Grad u\|_{L^2(\bb R^n)}$. Moreover, when $p = \frac{2n}{2n - \mu}$ and $q = \frac{2n}\mu$ in Theorem \ref{theorem:hls_inequality} we denote the corresponding optimal constant $\mc H(n, \mu, \frac{2n}{2n - \mu})$ by $\mc H$. That is, 
\begin{equation}
\label{eq:sharp_HLS_constant}
	\mc H
	= \sup\left\{\|I_\mu [f]\|_{L^{\frac{2n}\mu}(\bb R^n)}: f\in L^{\frac{2n}{2n - \mu}}(\bb R^n)\text{ and }\|f\|_{L^{\frac{2n}{2n - \mu}}(\bb R^n)} = 1\right\}. 
\end{equation}
It is well-known that $\mc S$ depends only on $n$ and that $\mc H$ depends only on $n$ and $\mu$. It is also well-known that for any $(x_0, \epsilon)\in \bb R^n\times (0, \infty)$, the infimum in \eqref{eq:sharp_sobolev_constant} is attained by 
\begin{equation}
\label{eq:normalized_bubbles}
	u_{x_0, \epsilon}(x)
	= c_n\left(\frac{\epsilon}{\epsilon^2 + |x- x_0|^2}\right)^{\frac{n -2}2}, 
\end{equation}
where $c_n>0$ is a constant for which $\|u_{x_0, \epsilon}\|_{L^{2^*}(\bb R^n)} = 1$ for all $(x_0, \epsilon)\in \bb R^n\times (0, \infty)$. From Lemma 1.2 of \cite{GaoYang2018} the quantity $\mc S_{HL}$ defined by 
\begin{equation}
\label{eq:sharp_choquard_constant}
	\mc S_{HL}^{-1} 
	= \inf\left\{\frac{\|u\|_{D^{1, 2}(\bb R^n)}^2}{\left(\int_{\bb R^n}I_\mu[|u|^{2^*_\mu}]|u|^{2^*_\mu} \right)^{1/2^*_\mu}}: u\in D^{1, 2}(\bb R^n)\setminus\{0\}\right\} 
\end{equation}
is well-defined and given by 
\begin{equation}
\label{eq:SHL_expression}
	\mc S_{HL} = \mc S\mc H^{1/2^*_\mu}. 
\end{equation}
Moreover, $u\in D^{1, 2}(\bb R^n)$ attains the infimum in the definition of $\mc S_{HL}^{-1}$ if and only if there is $(x_0, \epsilon)\in \bb R^n\times (0, \infty)$ such that $u$ is a non-zero constant multiple of $u_{x_0, \epsilon}$. In what follows, we formulate an analog of the sharp constant in \eqref{eq:sharp_choquard_constant} in the setting of vector-valued functions and for $2^*_\mu$ positively homogeneous nonlinearities $H\in C^0(\bb R^d)$. For any $U\in D^{1, 2}(\bb R^n; \bb R^d)$ recalling the definition of $\Psi$ in \eqref{eq:intro_Psi} then applying H\"older's inequality, the HLS inequality and Lemma \ref{lemma:homogeneous_integrability} gives
\begin{equation}
\label{eq:Psi(U)_upper_bound}
\begin{split}
	\Psi(U)
	& 
	\leq \|I_\mu[H(U)]\|_{L^{2n/\mu}(\bb R^n)}\|H(U)\|_{L^{2n/(2n - \mu)}(\bb R^n)}\\
	& \leq \mc H \|H(U)\|_{L^{2n/(2n - \mu)}(\bb R^n)}^2\\
	& \leq \mc H M_H^2\|U\|_{L^{2^*}(\bb R^n; \bb R^d)}^{2\cdot 2^*_\mu}. 
\end{split}
\end{equation}
Moreover, from Minkowski's inequality and the sharp Sobolev inequality we have 
\begin{equation}
\label{eq:from_minkowski_sobolev}
\begin{split}
	\|U\|_{L^{2^*}(\bb R^n; \bb R^d)}^2
	& = \left(\int_{\bb R^n}\left(\sum_{j = 1}^d|u_j|^2\right)^{2^*/2}\right)^{2/2^*}\\
	& \leq \sum_{j = 1}^d\left(\int_{\bb R^n}|u_j|^{2^*}\right)^{2/2^*}\\
	& \leq \mc S \sum_{j= 1}^d\int_{\bb R^n}|\Grad u_j|^2\\
	& = \mc S\|U\|_{D^{1, 2}(\bb R^n; \bb R^d)}^2.
\end{split}
\end{equation}
Combining \eqref{eq:Psi(U)_upper_bound} and \eqref{eq:from_minkowski_sobolev} gives
\begin{equation}
\label{eq:Psi_embedding}
	\Psi(U)
	\leq \mc HM_H^2\mc S^{2^*_\mu}\|U\|_{D^{1, 2}(\bb R^n; \bb R^d)}^{2\cdot2^*_\mu}. 
\end{equation}
This computation 
shows that the quantity $\mc N(H) = \mc N(H, n, \mu)$ defined by 
\begin{equation}
\label{eq:energy_threshold}
	\mc N(H)^{-1} 
	= \inf\left\{\frac{\|U\|_{D^{1, 2}(\bb R^n; \bb R^d)}^2}{\Psi(U)^{1/2^*_\mu}}: U\in D^{1, 2}(\bb R^n; \bb R^d)\setminus\{0\}\right\}	
\end{equation}
is well-defined and satisfies 
\begin{equation}
\label{eq:NHinv_lower_bound}
	\mc N(H)^{-1}
	\geq (\mc HM_H^2\mc S^{2^*_\mu})^{-1/2^*_\mu}
	= M_H^{-2/2^*_\mu}\mc S_{HL}^{-1},  
\end{equation}
where the equality holds in view of \eqref{eq:SHL_expression}. 
The following lemma both guarantees that equality holds in this inequality and classifies the extremal functions in the definition of $\mc N(H)$.  
\begin{lemma}
\label{lemma:H_sharp_constant}
If $\mu\in (0, n)$ and $H\in C^0(\bb R^d)$ is positively homogeneous of degree $2^*_\mu$ then
\begin{equation}
\label{eq:NHinv_equality}
	\mc N(H)
	= M_H^{2/2^*_\mu}\mc S_{HL}. 
\end{equation}
Moreover, $U\in D^{1, 2}(\bb R^n; \bb R^d)$ attains the infimum in the definition of $\mc N(H)^{-1}$ if and only if $U = cu_{x_0, \epsilon}s$ for some $c>0$, some $(x_0, \epsilon)\in \bb R^n\times (0, \infty)$ and some $s \in \bb S^{d -1}$ for which $H(s)= M_H$. 
\end{lemma}
\begin{proof}
Choose $s\in \bb S^{d - 1}$ for which $H(s) = M_H$, set $u(x) = c_n^{-1}u_{0, 1}(x) = (1 + |x|^2)^{-(n - 2)/2}$ as in \eqref{eq:normalized_bubbles} and define $U\in D^{1, 2}(\bb R^n; \bb R^d)$ by $U = (us_1, \ldots, us_d)$. This function satisfies both $\|U\|_{D^{1, 2}(\bb R^n; \bb R^d)}= \|u\|_{D^{1, 2}(\bb R^n)}$ and
\begin{equation*}
\begin{split}
	\Psi(U)
	& 
	= M_H^2\int_{\bb R^n}I_\mu[u^{2^*_\mu}]u^{2^*_\mu}. 
\end{split}
\end{equation*}
Therefore, since $u$ attains $\mc S_{HL}^{-1}$, 
\begin{equation*}
	\frac{\|U\|_{D^{1, 2}(\bb R^n; \bb R^d)}^2}{\Psi(U)^{1/2^*_\mu}}
	= (M_H^{2/2^*_\mu}\mc S_{HL})^{-1}. 
\end{equation*}
When combined with \eqref{eq:NHinv_lower_bound}, this equality establishes \eqref{eq:NHinv_equality}. Next suppose $U= (u_1, \ldots, u_d)\in D^{1, 2}(\bb R^n; \bb R^d)$ attains the infimum in the definition of $\mc N^{-1}(H)$. Evidently $U\not\equiv 0$ and estimating as in \eqref{eq:Psi(U)_upper_bound} and \eqref{eq:from_minkowski_sobolev} and using \eqref{eq:NHinv_equality} we obtain 
\begin{equation*}
\begin{split}
	\mc N\|U\|_{D^{1, 2}(\bb R^n; \bb R^d)}^2
	& = \Psi(U)^{1/2^*_\mu}\\
	& \leq(\mc HM_H^2)^{1/2^*_\mu}\mc S\|U\|_{D^{1, 2}(\bb R^n; \bb R^d)}^2\\
	& = \mc N\|U\|_{D^{1, 2}(\bb R^n; \bb R^d)}^2. 
\end{split}
\end{equation*}
Thus, for this $U$ equality holds throughout both of \eqref{eq:Psi(U)_upper_bound} and \eqref{eq:from_minkowski_sobolev}. Let $A = \{i\in \{1, \ldots, d\}: u_i\not\equiv 0\}$, let $B = \{i\in \{1, \ldots, d\}: u_i\equiv 0\}$, and note that $A\neq\emptyset$. We assume without losing generality that $A = \{1, \ldots, m\}$ for some $m\in \{1, \ldots, d\}$. Since equality holds in the application of Minkowski's inequality in \eqref{eq:from_minkowski_sobolev} 
we deduce the existence of $u\in D^{1, 2}(\bb R^n)\setminus\{0\}$ and, for each $i\in A$, the existence of a positive constant $a_i$ such that $u_i = a_i u$. Combining this fact with the fact that equality holds in the second inequality of \eqref{eq:from_minkowski_sobolev} we deduce that $u$ is an extremal function for the sharp Sobolev inequality. 
In particular there is $b>0$ and $(x_0, \epsilon)\in \bb R^n\times(0, \infty)$ such that $u(x)= bu_{x_0, \epsilon}$, where $u_{x_0, \epsilon}$ is as in \eqref{eq:normalized_bubbles}. Therefore, for each $i\in A$, upon relabeling $ba_i$ as simply $a_i$ we have
\begin{equation}
\label{eq:the_ui_cone}
	u_i(x)
	= a_iu_{x_0, \epsilon}.  
\end{equation}
Set $a = (a_1, \ldots, a_m, 0, \ldots, 0)\in \bb R^d\setminus\{0\}$ and note that \eqref{eq:the_ui_cone} implies $\frac{U}{|U|} \equiv \frac a{|a|}$ for all $x\in \bb R^n$. Since equality holds in the last inequality of \eqref{eq:Psi(U)_upper_bound} and by the $2^*_\mu$-homogenity of $H$ we have
\begin{equation*}
\begin{split}
    0
    & = \left(M_H^2\|U\|_{2^*}^{2\cdot 2^*_\mu}\right)^{\frac{n}{2n - \mu}} - \int_{\bb R^n}H(U)^{\frac{2n}{2n - \mu}}\\
    & = \int_{\bb R^n}\left(M_H^{\frac{2n}{2n - \mu}} - H\left(\frac{U}{|U|}\right)^{\frac{2n}{2n - \mu}}\right)|U|^{2^*}\\
    & = \int_{\bb R^n}\left(M_H^{\frac{2n}{2n - \mu}} - H\left(\frac{a}{|a|}\right)^{\frac{2n}{2n - \mu}}\right)|U|^{2^*}, 
\end{split}
\end{equation*}
from which we deduce that $M_H = H\left(\frac{a}{|a|}\right)$. In particular $U = c u_{x_0, \epsilon}s$ with $c = |a|> 0$ and with $s = \frac{a}{|a|}\in \bb S^{d - 1}$ satisfying $H(s) = M_H$. 
\end{proof}
\section{Sufficient Condition for Existence of a Minimizer}
\label{s:sufficient_condition_for_minimizer}
In this section we establish a sufficient condition for the existence of a minimizer of the functional $\Phi_{F}$ defined in \eqref{eq:intro_Phi_AF} subject to the constraint $\mc M$ defined in \eqref{eq:the_constraint}. We define 
\begin{equation}
\label{eq:energy_quotient}
	Q(U) = Q_{F, H, \mu}(U)
	= \frac{\Phi_{F}(U)}{\Psi(U)^{1/2_\mu^*}}
	\qquad \text{ for }U\in H_0^1(\Omega; \bb R^d)\setminus \{0\} 
\end{equation}
and we observe that when $F$ and $H$ satisfy \eqref{assumptions:FH}, Lemma \ref{lemma:Phi_AF_coercive} and inequality \eqref{eq:Psi_embedding} combine to guarantee that $Q$ is bounded below by a positive constant. For such $F$ and $H$, the quantity
\begin{equation*}
\begin{split}
	K(F, H)^{-1} 
	& = \inf\{\Phi_{F}(U): U\in \mc M\}\\
	& = \inf\{Q(U): U\in H_0^1(\Omega; \bb R^d)\setminus\{0\}\}, 
\end{split}
\end{equation*}
is well-defined and positive. For ease of notation, in what follows we do not indicate the dependence on $F$ or $H$ in the notations for $\Phi$, $Q$, $\mc N$ or $K$. For example, the notation $\Phi$ is understood to stand for $\Phi_{F}$. The following proposition is the main result of this section. 
\begin{prop}[Sufficient condition for existence of a minimizer]
\label{prop:sufficient_condition_for_minimizer}
Let $n\geq 3$, let $\Omega\subset \bb R^n$ be a bounded open set, and let $\mu\in (0, n)$. Suppose $F$ and $H$ satisfy assumptions \eqref{assumptions:FH}. If 
\begin{equation}
\label{eq:energy_threshhold}
	\mc N(H)< K(F,H)
\end{equation} 
then there is $U\in \mc M$ for which $\Phi(U) = K(F, H)^{-1}$. 
\end{prop}
Up to a nonzero constant multiple, every constrained minimizer of $\Phi$ is a nontrivial solution to problem \eqref{eq:main_problem}, so we obtain the following corollary to Proposition \ref{prop:sufficient_condition_for_minimizer}. 
\begin{coro}
\label{coro:existence_of_solution}
Under the hypotheses of Proposition \ref{prop:sufficient_condition_for_minimizer}, if $f$ and $h$ are as in \eqref{eq:the_gradients} then problem \eqref{eq:main_problem} admits a nontrivial weak solution. 
\end{coro}

The remainder of this section will be devoted to the proof of Proposition \ref{prop:sufficient_condition_for_minimizer}. This will be accomplished with the aid of a series of lemmata. The first such lemma is a result in the spirit of the classical Brezis-Lieb Lemma \cite{BrezisLieb1983}. 
\begin{lemma}
\label{lemma:BL_type}
Let $n\geq 3$, let $\mu\in (0, n)$, and let $H\in C^1(\bb R^d)$ be positively homogeneous of degree $2^*_\mu$. If $(U^k)_{k = 1}^\infty\subset L^{2^*}(\bb R^n; \bb R^d)$ is bounded in $L^{2^*}(\bb R^n; \bb R^d)$ and if $U^k(x)\to U(x)$ for a.e.\ $x\in \bb R^n$ then $U\in L^{2^*}(\bb R^n; \bb R^d)$ and $\lim_k (\Psi(U^k) - \Psi(U^k - U))=\Psi(U)$. 
\end{lemma}
\begin{proof}
The containment $U\in L^{2^*}(\bb R^n; \bb R^d)$ is a consequence of Fatou's lemma. 
Lemma \ref{lemma:homogeneous_integrability} guarantees that for every $k$ there holds
\begin{equation*}
	\|H(U^k)\|_{\frac{2n}{2n - \mu}} 
	\leq M_{H}\|U^k\|_{L^{2^*}(\bb R^n; \bb R^d)}^{2^*_\mu}. 
\end{equation*}
In particular, $(H(U^k))_{k = 1}^\infty$ is bounded in $L^{\frac{2n}{2n - \mu}}(\bb R^n)$. By the continuity of $H$ and the a.e.\ convergence $U^k\to U$, we have $H(U^k)\to H(U)$ a.e.\ in $\bb R^n$. Since $(H(U^k))_{k = 1}^\infty$ is bounded in $L^{2n/(2n - \mu)}(\bb R^n)$ and $H(U^k)\to H(U)$ a.e.\ in $\bb R^n$ we have that $H(U^k)\weakconv H(U)$ weakly in $L^{2n/(2n - \mu)}(\bb R^n)$. 
By a similar argument, $(H(U^k - U))_{k = 1}^\infty$ is bounded in $L^{2n/(2n- \mu)}(\bb R^n)$ and $H(U^k- H)\to 0$ a.e.\ in $\bb R^n$, so $H(U^k - U)\weakconv H(0) = 0$ weakly in $L^{2n/(2n -\mu)}(\bb R^n)$. In particular, since $I_\mu[H(U)]\in L^{2n/\mu}(\bb R^n)$, we have 
\begin{equation}
\label{eq:the_weak_convergence1}
	\int_{\bb R^n}I_\mu[H(U)]H(U^k)  \to \int_{\bb R^n}I_\mu[H(U)]H(U) 
\end{equation}
and, using the symmetry of $I_\mu$, we have
\begin{equation}
\label{eq:the_weak_convergence2}
	\int_{\bb R^n}I_\mu[H(U^k - U)]H(U)
	= \int_{\bb R^n}I_\mu[H(U)]H(U^k - U) \to 0.
\end{equation}
Now, from \eqref{eq:the_weak_convergence1}, \eqref{eq:the_weak_convergence2} and the symmetry of $I_\mu$ we have
\begin{equation}
\label{eq:brezis_lieb_reduction}
\begin{split}
	\Psi(U^k)&  - \Psi(U^k - U)\\
	= &\; \int_{\bb R^n}I_\mu[H(U^k) - H(U^k - U) - H(U)]H(U^k)\\
	& + \int_{\bb R^n}I_\mu[H(U^k - U)](H(U^k) - H(U^k - U) - H(U))\\
	& + \int_{\bb R^n}I_\mu[H(U)]H(U^k)+ \int_{\bb R^n}I_\mu[H(U^k - U)]H(U)\\
	= &\; \int_{\bb R^n}I_\mu[H(U^k) + H(U^k - U)](H(U^k) - H(U^k - U) - H(U))\\
	& + \Psi(U)+ \circ(1). 
\end{split}
\end{equation}
The fact that both of $(H(U^k))_{k = 1}^\infty$ and $(H(U^k - U))_{k = 1}^\infty$ are bounded in $L^{2n/(2n -\mu)}(\bb R^n)$ together with the HLS inequality guarantees that $(I_\mu[H(U^k) + H(U^k - U)])_{k = 1}^\infty$ is bounded in $L^{2n/\mu}(\bb R^n)$. Therefore, in view of \eqref{eq:brezis_lieb_reduction}, and H\"older's inequality, to complete the proof the Lemma it suffices to show that 
\begin{equation}
\label{eq:brezis_lieb_last_step}
	H(U^k) - H(U^k - U) \to  H(U) \qquad \text{ in }L^{2n/(2n - \mu)}(\bb R^n). 
\end{equation}
To do so we first observe that for any $\epsilon>0$ there is $C(\epsilon) = C(\epsilon, n, \mu, \max_{\bb S^{d - 1}}|\Grad H|)> 0$ such that for any $a, b\in \bb R^d$,
\begin{equation}
\label{eq:initial_H_difference}
	|H(a + b) - H(a)|
	\leq \epsilon|a|^{2^*_\mu} + C(\epsilon)|b|^{2^*_\mu}. 
\end{equation}
Fixing $\epsilon>0$ and applying \eqref{eq:initial_H_difference} with $a = U^k -U$ and $b = U$ gives
\begin{equation*}
	|H(U^k) - H(U^k - U)|
	\leq \epsilon|U^k - U|^{2^*_\mu} + C(\epsilon)|U|^{2^*_\mu}. 
\end{equation*}
Using the notation $t^+ = \max\{t, 0\}$ for $t\in \bb R$ and setting 
\begin{equation*}
	f_k^\epsilon
	= \left(|H(U^k) - H(U^k - U) - H(U)| - \epsilon|U^k -U|^{2^*_\mu}\right)^+
\end{equation*}
we have $|f_k^\epsilon|\leq (C(\epsilon) + M_{H})|U|^{2^*_\mu}$ and therefore, 
\begin{equation*}
	|f_k^\epsilon|^{\frac{2n}{2n - \mu}}
	\leq (C(\epsilon) + M_{H})^{\frac{2n}{2n -\mu}}|U|^{2^*}
	\in L^1(\bb R^n). 
\end{equation*}
Since in addition, $f_k^\epsilon\to 0$ a.e.\ in $\bb R^n$, the Dominated Convergence Theorem guarantees that for all $\epsilon>0$, $\|f_k^\epsilon\|_{L^{2n/(2n - \mu)}(\bb R^n)}\to 0$ as $k\to \infty$. Finally, since
\begin{equation*}
	|H(U^k)  - H(U^k - U) - H(U)|
	\leq f_k^\epsilon + \epsilon|U^k - U|^{2^*_\mu}, 
\end{equation*}
Minkowski's inequality gives
\begin{equation*}
\begin{split}
	\||H(U^k) &  - H(U^k - U) - H(U)\|_{L^{\frac{2n}{2n - \mu}}(\bb R^n)}\\
	& 
	\leq \|f_k^\epsilon\|_{L^{\frac{2n}{2n - \mu}}(\bb R^n)} + \epsilon\||U^k - U|^{2^*_\mu}\|_{L^{\frac{2n}{2n - \mu}}(\bb R^n)}\\
	& = \circ(1) + \epsilon\sup_k\|U^k - U\|_{L^{2^*}(\bb R^n)}^{2^*_\mu}. 
\end{split}
\end{equation*}
Since $\epsilon> 0$ is arbitrary, and $(U_k - U)_{k = 1}^\infty$ is bounded in $L^{2^*}(\bb R^n)$, the convergence in \eqref{eq:brezis_lieb_last_step} is established. 
\end{proof}
\begin{proof}[Proof of Proposition \ref{prop:sufficient_condition_for_minimizer}]
Let $(U^k)_{k = 1}^\infty\subset \mc M$ be a sequence for which $\Phi(U^k)\to K^{-1}$. Lemma \ref{lemma:Phi_AF_coercive} ensures that $\Phi$ is coercive, so $(U^k)_{k = 1}^\infty$ is bounded in $H_0^1(\Omega; \bb R^d)$. The reflexivity of $H_0^1(\Omega; \bb R^d)$ and the compactness of the embedding $H_0^1(\Omega; \bb R^d)\hookrightarrow L^2(\Omega; \bb R^d)$ guarantees the existence of $U\in H_0^1(\Omega; \bb R^d)$ and a subsequence of $(U^k)_{k = 1}^\infty$ along which all of the following hold:
\begin{equation*}
\begin{split}
	U^k\weakconv U& \text{ weakly in }H_0^1(\Omega; \bb R^d)\\
	U^k\to U & \text{ in }L^2(\Omega; \bb R^d)\\
	U^k\to U & \text{ a.e.\ in }\Omega. 
\end{split}
\end{equation*}
The continuity of $F$ and the a.e.\ convergence $U^k\to U$ ensures that $F(U^k)\to F(U)$ a.e.\ in $\Omega$. Moreover, Lemma \ref{lemma:properties_homogeneous_functions} \ref{item:homogeneous_upper_bound} gives $|F(U^k)|\leq M_{|F|} |U^k|^2$, so from the convergence $U^k\to U$ in $L^2(\Omega; \bb R^d)$ and the Generalized Dominated Convergence Theorem we obtain $\int_\Omega F(U^k)\to \int_\Omega F(U)$. Combining this convergence with the fact that the weak convergence $U^k\weakconv U$ in $H_0^1(\Omega; \bb R^d)$ guarantees
\begin{equation*}
	\|U^k\|_{H_0^1(\Omega; \bb R^d)}^2 - \|U\|_{H_0^1(\Omega; \bb R^d)}^2
	= \|U^k- U\|_{H_0^1(\Omega; \bb R^d)}^2 + \circ(1),
\end{equation*}
we obtain 
\begin{equation}
\label{eq:minimizing_sequence_consequence1}
\begin{split}
	K^{-1} + \circ(1)
	& = \|U^k\|_{H_0^1(\Omega; \bb R^d)}^2 - \int_\Omega F(U^k)\\
	& = \Phi(U) + \|U^k\|_{H_0^1(\Omega; \bb R^d)}^2 - \|U\|_{H_0^1(\Omega; \bb R^d)}^2 + \circ(1)\\
	& = \Phi(U) + \|U^k- U\|_{H_0^1(\Omega; \bb R^d)}^2 + \circ(1)\\
	& \geq K^{-1}\Psi(U)^{1/2^*_\mu} + \|U^k- U\|_{H_0^1(\Omega; \bb R^d)}^2 + \circ(1).\\
\end{split}
\end{equation}
Now from the definition of $\mc N$, and Lemma \ref{lemma:BL_type} 
 we have
\begin{equation*}
\begin{split}
	\mc N\|U^k- U\|_{H_0^1(\Omega; \bb R^d)}^2
	& \geq \Psi(U^k - U)^{1/2^*_\mu}\\
	& = \left(\Psi(U^k) -\Psi(U)\right)^{1/2^*_\mu} + \circ(1)\\
	& \geq \left(\Psi(U^k)^{1/2^*_\mu} -\Psi( U)^{1/2^*_\mu}\right)+ \circ(1)\\
	& = \left(1 -\Psi( U)^{1/2^*_\mu}\right)+ \circ(1), 
\end{split}
\end{equation*}
where the assumption $U^k\in \mc M$ was used in the final equality. Bringing this back into estimate \eqref{eq:minimizing_sequence_consequence1} then letting $k\to\infty$ gives
\begin{equation*}
	0 
	\leq(K^{-1}- \mc N^{-1})\left(1 - \Psi(U)^{1/2^*_\mu}\right).  
\end{equation*}
Combining this estimate with the assumption $K> \mc N$ yields $\Psi(U)\geq 1$. Bringing this inequality back into estimate \eqref{eq:minimizing_sequence_consequence1} shows that $U^k\to U$ in $H_0^1(\Omega ;\bb R^d)$, so the continuity of $\Phi$ on $H_0^1(\Omega; \bb R^d)$ gives 
\begin{equation*}
	K^{-1} = \lim_k\Phi(U^k) = \Phi(U).
\end{equation*} 
To complete the proof of the proposition, it remains to show that $U\in \mc M$. Since $U^k\to U$ in $H_0^1(\Omega; \bb R^d)$, inequality \eqref{eq:Psi_embedding} guarantees that $\Psi(U^k- U)\to 0$. Therefore, Lemma \ref{lemma:BL_type} gives
\begin{equation*}
	\Psi(U)
	= \lim_k\left(\Psi(U^k) - \Psi(U^k - U)\right) 
	= 1. 
\end{equation*}
\end{proof}
%
\section{Existence of Solutions}
\label{s:existence}
This section is devoted to the proof of Theorem \ref{theorem:existence}. The strategy of the proof is to verify that the energy smallness condition in \eqref{eq:energy_threshhold} is satisfied. This is accomplished by constructing $U\in H_0^1(\Omega; \bb R^d)\setminus\{0\}$ for which $Q(U)< \mc N^{-1}$, where $Q$ is as in \eqref{eq:energy_quotient} and $\mc N$ is as in \eqref{eq:energy_threshold}.  
\begin{proof}[Proof of Theorem \ref{theorem:existence}]
We assume with no loss of generality that $0\in \Omega$ and we choose $\delta>0$ for which $B_{4\delta}\subset\Omega$. Let $\eta\in C_c^\infty(\bb R^n; [0, 1])$ satisfy both $\eta\equiv 1$ on $B_\delta$ and $\eta\equiv 0$ on $\bb R^n\setminus B_{2\delta}$. For $\epsilon>0$ define 
\begin{equation*}
	w_\epsilon = \eta u_\epsilon, 
\end{equation*}
where $u_\epsilon = u_{0, \epsilon}$ and $u_{0, \epsilon}$ is as in equation \eqref{eq:normalized_bubbles}. In particular, by the choice of $c_n$ in \eqref{eq:normalized_bubbles}, for every $\epsilon>0$ we have $\|u_\epsilon\|_{L^{2^*}(\bb R^n)} = 1$. The functions $w_\epsilon$ satisfy the estimates
\begin{equation}
\label{eq:cutoff_bubble_grad_norm}
	\|\Grad w_\epsilon\|_2^2 
	= \mc S^{-1} + O(\epsilon^{n - 2})\\
\end{equation}
and
\begin{equation}
\label{eq:cutoff_bubble_2norm}
	\|w_\epsilon\|_2^2 
	= \begin{cases}
	b_{n}\epsilon^2 + O(\epsilon^{n -2})& \text{ if }n>4\\
	b_{4}\epsilon^2|\log\epsilon| + O(\epsilon^2) & \text{ if }n =4, 
	\end{cases}
\end{equation}
where 
\begin{equation}
\label{eq:bnp}
	b_{n}
	= \begin{cases}
	c_{n}^2|\bb S^{n - 1}|\frac{\Gamma\left(\frac{n - 4}2\right)\Gamma\left(\frac{n}2\right)}{2\Gamma(n -2)} & \text{ if }n > 4\\
	c_{4}^2|\bb S^{n - 1}| & \text{ if }n = 4, 
	\end{cases}
\end{equation}
and $\mc S  = \mc S(n)$ is the sharp Sobolev constant defined in equation \eqref{eq:sharp_sobolev_constant}, see for example \cite{GarciaPeral1987}. By hypothesis there is $\theta\in \bb S^{d - 1}$ for which both $H(\theta) = M_H$ and $F(\theta)>0$. Fix any such $\theta$ and define $U^\epsilon \in H_0^1(\Omega; \bb R^d)$ by
\begin{equation*}
	U^\epsilon 
	= (w_\epsilon\theta_1, \ldots,w_\epsilon\theta_d). 
\end{equation*}
We proceed to estimate $Q(U^\epsilon)$ as $\epsilon\to 0$, where  $Q$ is the quotient defined in \eqref{eq:energy_quotient}. First, we have
\begin{equation}
\label{eq:Phi_U_epsilon}
\begin{split}
	\Phi(U^\epsilon)
	& 
	= \|\Grad w_\epsilon\|_{L^2(\bb R^n)}^2 - F(\theta)\|w_\epsilon\|_{L^2(\bb R^n)}^2. 
\end{split}
\end{equation}
Next we estimate $\Psi(U^\epsilon)$ as $\epsilon\to 0$. By the degree-$2^*_\mu$ homogeneity of $H$ and the choice of $\theta$ we have
\begin{equation}
\label{eq:epsilon_denom_initial}
\begin{split}
	\Psi(U^\epsilon)
	& 
	= M_H^2\int_{\bb R^n}I_\mu[w_\epsilon^{2^*_\mu}]w_\epsilon^{2^*_\mu}. 
\end{split}
\end{equation}
Define $I^\epsilon$ and $J^\epsilon$ by 
\begin{equation*}
\begin{split}
	I^\epsilon & = \int_{\bb R^n\setminus B_\delta}\int_{B_\delta}\frac{u_\epsilon^{2^*_\mu}(y)u_\epsilon^{2^*_\mu}(x)}{|x- y|^\mu}\; \d y\; \d x\\
	J^\epsilon& = \int_{\bb R^n\setminus B_\delta}\int_{\bb R^n\setminus B_\delta}\frac{u_\epsilon^{2^*_\mu}(y)u_\epsilon^{2^*_\mu}(x)}{|x- y|^\mu}\; \d y\; \d x. 
\end{split}
\end{equation*}
Using \eqref{eq:epsilon_denom_initial}, the fact that $\supp w_\epsilon \subset \bar B_{2\delta}$, and the fact that $w_\epsilon \equiv u_\epsilon$ on $B_\delta$ we have
\begin{equation}
\label{eq:decompose_double_integral}
\begin{split}
	M_H^{-2}\Psi(U^\epsilon)
	& = \int_{B_{2\delta}}\int_{B_{2\delta}}\frac{w_\epsilon^{2^*_\mu}(y)w_\epsilon^{2^*_\mu}(x)}{|x- y|^\mu}\; \d y\; \d x\\
	& \geq \int_{B_\delta}\int_{B_\delta}\frac{u_\epsilon^{2^*_\mu}(y)u_\epsilon^{2^*_\mu}(x)}{|x- y|^\mu}\; \d y\; \d x\\
	& = \int_{\bb R^n}\int_{\bb R^n}\frac{u_\epsilon^{2^*_\mu}(y)u_\epsilon^{2^*_\mu}(x)}{|x- y|^\mu}\; \d y\; \d x - 2I^\epsilon - J^\epsilon\\
	& = \left(\mc S_{HL}\|u_\epsilon\|_{D^{1, 2}(\bb R^n)}^2\right)^{2^*_\mu} - 2I^\epsilon - J^\epsilon\\
	& = \left(\frac{\mc S_{HL}}{\mc S}\right)^{2^*_\mu} - 2I^\epsilon - J^\epsilon, 
\end{split}
\end{equation}
where the penultimate line holds as $u_\epsilon$ attains the infimum in \eqref{eq:sharp_choquard_constant} and the final line holds since $u_\epsilon$ attains the infimum in \eqref{eq:sharp_sobolev_constant}. We proceed to estimate $I^\epsilon$ and $J^\epsilon$ from above. To estimate $I^\epsilon$, we first apply the HLS inequality to obtain 
\begin{equation}
\label{eq:I_estimate_apply_HLS}
\begin{split}
	I^\epsilon
	& = \int_{\bb R^n}\int_{\bb R^n}\frac{u_\epsilon^{2^*_\mu}(x)\chi_{\bb R^n\setminus B_\delta}(x)\cdot u_\epsilon^{2^*_\mu}(y)\chi_{B_\delta}(y)}{|x- y|^\mu}\; \d y\; \d x\\
	& \leq \mc H\|u_\epsilon\|_{L^{2^*}(\bb R^n\setminus B_\delta)}^{2^*_\mu}\|u_\epsilon\|_{L^{2^*}(B_\delta)}^{2^*_\mu}. 
\end{split}
\end{equation}
Moreover, we have both 
\begin{equation}
\label{eq:2*_norm_outside_ball}
\begin{split}
	\|u_\epsilon\|_{L^{2^*}(\bb R^n\setminus B_\delta)}^{2^*}
	& = c_n^{2^*}\int_{\bb R^n\setminus B_\delta}\left(\frac{\epsilon}{\epsilon^2 + |x|^2}\right)^n\; \d x\\
	& \leq c_n^{2^*}\epsilon^n\int_{\bb R^n\setminus B_\delta}|x|^{-2n}\; \d x\\
	& \leq C(n)\left(\frac\epsilon \delta\right)^n
\end{split}
\end{equation}
and, using the change of variable $z = x/\epsilon$, 
\begin{equation}
\label{eq:2*_norm_inside_ball}
\begin{split}
	\|u_\epsilon\|_{L^{2^*}(B_\delta)}^{2^*}
	& \leq \|u_\epsilon\|_{L^{2^*}(\bb R^n)}^{2^*}\\
	& = c_n^{2^*}\int_{\bb R^n}(1 + |z|^2)^{-n}\; \d z\\
	& \leq C(n). 
\end{split}
\end{equation}
Bringing estimates \eqref{eq:2*_norm_outside_ball} and \eqref{eq:2*_norm_inside_ball} back into estimate \eqref{eq:I_estimate_apply_HLS} gives
\begin{equation}
\label{eq:final_I_estimate}
	I^\epsilon
	\leq C(n, \mu)\left(\frac\epsilon\delta\right)^{\frac{2n - \mu}2}. 
\end{equation}
The estimate of $J^\epsilon$ is performed in a similar spirit to the estimate of $I^\epsilon$. In particular, using the HLS inequality and estimate \eqref{eq:2*_norm_outside_ball} one obtains
\begin{equation}
\label{eq:final_J_estimate}
\begin{split}
	J^\epsilon
	& = \int_{\bb R^n}\int_{\bb R^n}\frac{u_\epsilon^{2^*_\mu}(x)\chi_{\bb R^n\setminus B_\delta}(x)\cdot u_\epsilon^{2^*_\mu}(y)\chi_{\bb R^n\setminus B_\delta}(y)}{|x- y|^\mu}\; \d y\; \d x\\
	& \leq \mc H\|u_\epsilon\|_{L^{2^*}(\bb R^n\setminus B_\delta)}^{2\cdot 2^*_\mu}\\
	& \leq C(n, \mu)\left(\frac \epsilon \delta\right)^{2n - \mu}. 
\end{split}
\end{equation}
Bringing estimates \eqref{eq:final_I_estimate} and \eqref{eq:final_J_estimate} back into \eqref{eq:decompose_double_integral} and using Lemma \ref{lemma:H_sharp_constant}, if $\epsilon>0$ is sufficiently small we obtain
\begin{equation}
\label{eq:final_Psi_estimate}
\begin{split}
	\Psi(U^\epsilon)^{-1/2^*_\mu}
	& 
	\leq \frac{\mc S}{\mc N}\left(1 + C(n, \mu, \delta, M_H)\epsilon^{\frac{2n-\mu}2}\right).
\end{split}
\end{equation}
Therefore, using \eqref{eq:Phi_U_epsilon} and \eqref{eq:final_Psi_estimate} we obtain 
\begin{equation}
\label{eq:bring_estimates_to_Q}
\begin{split}
	Q(U^\epsilon)
	& 
	\leq \frac{\mc S}{\mc N}\left(\|\Grad w_\epsilon\|_{L^2(\bb R^n)}^2 - F(\theta)\|w_\epsilon\|_{L^2(\bb R^n)}^2\right)\left(1 + O(\epsilon^{\frac{2n - \mu}2})\right). 
\end{split}
\end{equation}
If $n> 4$ then using \eqref{eq:cutoff_bubble_grad_norm} and \eqref{eq:cutoff_bubble_2norm} in \eqref{eq:bring_estimates_to_Q} we obtain 
\begin{equation}
\label{eq:n_large_Q_estimate}
\begin{split}
	Q(U^\epsilon)
	& 
	\leq \mc N^{-1}\left(1 - b_n\mc S F(\theta)\epsilon^2 + \circ(\epsilon^2)\right).  
\end{split}
\end{equation}
Similarly, if $n = 4$ then using \eqref{eq:cutoff_bubble_grad_norm} and \eqref{eq:cutoff_bubble_2norm} in \eqref{eq:bring_estimates_to_Q} gives
\begin{equation}
\label{eq:n_small_Q_estimate}
\begin{split}
	Q(U^\epsilon)
	& 
	\leq \mc N^{-1}\left(1 - b_4\mc S F(\theta)\epsilon^2|\log\epsilon| + O(\epsilon^2)\right). 
\end{split}
\end{equation}
In view of the positivity of $F(\theta)$, estimates \eqref{eq:n_large_Q_estimate} and \eqref{eq:n_small_Q_estimate} guarantee that if $n\geq 4$ then the inequality $Q(U^\epsilon)< \mc N^{-1}$ holds whenever $\epsilon$ is sufficiently small. Corollary \ref{coro:existence_of_solution} now guarantees the existence of a nontrivial weak solution to problem \eqref{eq:main_problem}. 
\end{proof}
%
\section{Nonexistence of nontrivial solutions}
\label{s:non_existence}
This section is dedicated to the proof of Theorem \ref{theorem:nonexistence}. The proof is based on the following Poho\v zaev identity for solutions to \eqref{eq:main_problem}. 
\begin{lemma}[Poho\v zaev Identity]
\label{lemma:pohozaev}
Let $n\geq 3$, let $\Omega\subset \bb R^n$ be a bounded open set with $\bdy\Omega \in C^{1, 1}$ and let $\mu\in (0, n)$. Suppose $F$ and $H$ satisfy \ref{item:F_homogeneous} and \ref{item:H_positively_homogeneous} respectively and that $f$ and $h$ are as in \eqref{eq:the_gradients}. If $U\in H_0^1(\Omega; \bb R^d)$ is a weak solution to \eqref{eq:main_problem} then 
\begin{equation*}
    \int_\Omega F(U)\; \d x
    = \frac 12 \sum_{i = 1}^d\int_{\bdy\Omega}\left(\frac{\partial u_i}{\partial \nu}\right)^2 \lb x, \nu\rb \; \d S. 
\end{equation*}
\end{lemma}
The following lemma will be used in the proof of Lemma \ref{lemma:pohozaev}. It is a special case of Theorem 2 of \cite{Degiovanni2003} where the $C^2(\Omega)\cap C^1(\bar\Omega)$-smoothness assumptions of the general variational identity of Pucci and Serrin \cite{PucciSerrin1986} were relaxed. 
\begin{lemma}
\label{lemma:poissons_pohozaev}
Let $\Omega\subset\bb R^n$ be a bounded open set with $\bdy\Omega\in C^1$ and let $g\in C^0(\bar\Omega)$. If $u\in C^1(\bar\Omega)$ satisfies
\begin{equation*}
\begin{cases}
    -\lap u= g & \text{ in }\Omega\\
    u = 0 & \text{ on }\bdy \Omega
\end{cases}
\end{equation*}
in the weak sense then 
\begin{equation*}
	\frac{n - 2}{2}\int_\Omega |\Grad u|^2 + \int_\Omega g\lb x, \Grad u\rb
	= -\frac 12\int_{\bdy\Omega}\left(\frac{\partial u}{\partial \nu}\right)^2\lb x, \nu\rb\; \d S. 
\end{equation*}
\end{lemma}
\begin{proof}[Proof of Lemma \ref{lemma:pohozaev}]
For each $i\in \{1, \ldots, d\}$, $u_i\in H_0^1(\Omega)$ satisfies
\begin{equation}
\label{eq:ui_PDE}
    -\lap u_i = f_i(U) + I_\mu[H(U)]h_i(U)
\end{equation}
in the weak sense and Proposition \ref{prop:C1alpha_regularity} guarantees that $u_i\in C^1(\bar\Omega)$. Applying Lemma \ref{lemma:poissons_pohozaev} to $u_i$ with $g = g_i = f_i(U) + I_\mu[H(U)]h_i(U)$ then summing the resulting identities over $i$ and using both \eqref{eq:the_gradients} and Lemma \ref{lemma:properties_homogeneous_functions}\ref{item:homogeneous_grad_dot}, we have
\begin{equation}
\label{eq:apply_poissons_pohozaev}
\begin{split}
    -\frac{n - 2}2\sum_{i = 1}^d& \int_\Omega |\Grad u_i|^2 - \frac 12\sum_{i = 1}^d\int_{\bdy\Omega}\left(\frac{\partial u_i}{\partial \nu}\right)^2\lb x, \nu\rb\; \d S\\
    & = \sum_{i = 1}^d\int_\Omega \lb x, \Grad u_i\rb (f_i(U) + I_\mu[H(U)]h_i(U))\\
    & = \frac 12\int_\Omega \lb x, \Grad(F\circ U)\rb + \frac 1{2^*_\mu}\int_\Omega I_\mu[H(U)]\lb x, \Grad(H\circ U)\rb\\
    & = \frac{A}{2^*_\mu} - \frac n2\int_\Omega F(U),  
\end{split}
\end{equation}
where
\begin{equation*}
    A = \int_{\bb R^n}\int_{\bb R^n}\frac{v(y)\lb x, \Grad v(x)\rb}{|x - y|^\mu}\; \d y\;\d x
\end{equation*}
and $v\in C^1(\bb R^n)$ is defined by $v = (H\circ U)\chi_\Omega$. 
Using Fubini's theorem it is routinely verified that
\begin{equation}
\label{eq:two_A}
    2A = \int_{\bb R^n}\int_{\bb R^n}\frac{v(y)\lb x, \Grad v(x)\rb + v(x)\lb y, \Grad v(y)\rb}{|x - y|^\mu}\; \d y\;\d x. 
\end{equation}
Moreover, for any $t\in (\frac 12, 2)$, using the change of variable $x\mapsto tx$ and the change of variable $y\mapsto ty$ gives
\begin{equation*}
	\int_{\bb R^n}\int_{\bb R^n}\frac{v(x)v(y)}{|x - y|^\mu}\; \d y\; \d x
	= t^{2n - \mu}\int_{\bb R^n}\int_{\bb R^n}\frac{v(tx)v(ty)}{|x - y|^\mu}\; \d y\; \d x. 
\end{equation*}
Differenting this identity relative to $t$, evaluating the resulting equality at $t = 1$, and using \eqref{eq:two_A} gives
\begin{equation}
\label{eq:A_computed}
    A = -\frac{2n - \mu}{2}\int_{\bb R^n}\int_{\bb R^n}\frac{v(x)v(y)}{|x - y|^\mu}\; \d y\; \d x. 
\end{equation}
Independently, for each $i\in \{1, \ldots, d\}$, testing equation \eqref{eq:ui_PDE} against $u_i$, summing the resulting identities over $i$, and using both \eqref{eq:the_gradients} and Lemma \ref{lemma:properties_homogeneous_functions}\ref{item:homogeneous_grad_dot} gives
\begin{equation}
\label{eq:test_against_ui}
\begin{split}
    \sum_{i = 1}^d\int_\Omega |\Grad u_i|^2
    & 
    = \int_\Omega(F(U) + I_\mu[H(U)]H(U)). 
\end{split}
\end{equation}
Bringing \eqref{eq:A_computed} and \eqref{eq:test_against_ui} back to \eqref{eq:apply_poissons_pohozaev} yields the asserted equality. 
\end{proof}
\begin{proof}[Proof of Theorem \ref{theorem:nonexistence}]
In view of the translation invariance of problem \eqref{eq:main_problem}, we may assume with no loss of generality that $\Omega$ is star-shaped with respect to the origin and thus $\lb x, \nu\rb \geq 0$ for all $x\in \bdy \Omega$. For any weak solution $U\in H_0^1(\Omega; \bb R^d)$ to problem \eqref{eq:main_problem}, Lemma \ref{lemma:pohozaev} gives
\begin{equation*}
\begin{split}
    0
    & \leq \sum_{i = 1}^d\int_{\bdy \Omega}\left(\frac{\partial u_i}{\partial\nu}\right)^2\lb x, \nu\rb\; \d S\\
    & = 2\int_\Omega F(U)\\
    & \leq 2M_F\int_\Omega |U|^2\\
    & \leq 0,  
\end{split}
\end{equation*}
and hence equality holds throughout this string of inequalities. Since $M_F< 0$, we conclude that $\|U\|_{L^2(\Omega; \bb R^d)} = 0$. 
\end{proof}
\section{Appendix}
\label{s:appendix}
\begin{lemma}
\label{lemma:simple_construction}
Let $\Omega \subset \bb R^n$ be open. If $F\in C^0(\bb R^d)$ is homogeneous of degree $1$ and if $M_F> 0$ then there is $U\in H_0^1(\Omega; \bb R^d)$ such that $\int_\Omega F(U)> 0$. 
\end{lemma}
\begin{proof}[Proof of Lemma \ref{lemma:simple_construction}]
By the continuity of $F$ and the assumption $M_F>0$ there is $\sigma\in \bb S^{d - 1}$ for which $F(\sigma) = M_F$ and there is $\epsilon>0$ such that $F(s)> M_F/2$ for all $s\in B(\sigma, \epsilon)$. Let $\eta\in C_c^\infty(\bb R)$ satisfy
\begin{equation*}
	\eta(t) = \begin{cases}
	0 & \text{ if } t\leq 0\\
	1 & \text{ if }t\geq 1. 
	\end{cases}
\end{equation*}
For $B(x_0, r)\subset \Omega$, define $U\in C_c^\infty(\Omega; \bb R^d)$ by 
\begin{equation*}
	U(x) 
	= \left( 1- \eta\left(\frac{|x - x_0|^2}{r^2}\right)\right)\sigma. 
\end{equation*}
By the choice of $\sigma$ and the homogeneity of $F$, for any $x\in B(x_0, r)$ we have
\begin{equation*}
	F(U(x))
	= \left( 1- \eta\left(\frac{|x - x_0|^2}{r^2}\right)\right)^2M_F
	\geq 0. 
\end{equation*}
Moreover, there is $\delta\in(0, r)$ such that $U(x)\in B(\sigma, \epsilon)$ whenever $x\in B(x_0, \delta)$. For any such $\delta$, 
\begin{equation*}
	\int_\Omega F(U)\; \d x
	\geq \int_{B(x_0, \delta)}F(U)\; \d x
	\geq \frac{M_F}2|B(x_0, \delta)|
	> 0. 
\end{equation*}
\end{proof}
%

\subsection{Regularity}
\label{ss:regularity}
This subsection is dedicated to the proof of Proposition \ref{prop:C1alpha_regularity}, which will be accomplished with the aid of a series of lemmata. For $v\in \bb R$ and $m> 0$ let us denote the truncation at level $m$ of $v$ by 
\begin{equation}
\label{eq:level_m_truncation}
    \tau_m(v) = \begin{cases}
    -m & \text{ if }v< -m\\
    v & \text{ if }|v|\leq m\\
    m & \text{ if }v> m. 
    \end{cases}
\end{equation}
With this notation we have
\begin{equation*}
    |\tau_m(v)| = \tau_m(|v|) = \begin{cases}
    v& \text{ if }|v|\leq m\\
    m& \text{ if }|v|> m, 
    \end{cases}
\end{equation*}
and if $V = (v^1, \ldots, v^d)\in \bb R^d$ then
\begin{equation*}
   \tau_m(|V|) = \begin{cases}
        m & \text{ if }|V|> m\\
        |V|& \text{ if }|V|\leq m.
    \end{cases}
\end{equation*}
Later we will consider these truncation notations pointwise for functions $v:\Omega \to \bb R$ and $V:\Omega \to \bb R^d$. The following lemma relates powers of the level-$m$ turncation of the length of $V\in \bb R^d$ to the sum of powers of the level-$m$ truncations of the coordinates of $V$. Since the proof is elementary, we omit the details.  
\begin{lemma}
\label{lemma:vector_truncation}
For every $t\in [1, \infty)$ there is a constant $C_1 = C_1(d, t)>0$ such that for every $V\in \bb R^d$ and every $m>0$ there holds
\begin{equation}
\label{eq:vector_truncation_estimate}
    C_1^{-1}\tau_m(|V|)^t
    \leq \sum_{i = 1}^d|\tau_m(v_i)|^t
    \leq C_1\tau_m(|V|)^t. 
\end{equation}
\end{lemma}
The following lemma is adapted from Lemmata 3.2 and 3.3 of \cite{MorozVanSchaftingen2015}. 
\begin{lemma}
\label{lemma:moroz_van_interpolation}
Let $n\geq 3$, let $\Omega\subset \bb R^n$ be a bounded domain and let $\mu\in (0, n)$. For every $\theta\in (1 - \frac \mu n, 1 + \frac \mu n)$ and every $\epsilon>0$ there is a constant $C_\epsilon = C_\epsilon(n, \mu,\theta)>0$ such that for every pair of functions $P, Q\in L^{\frac{2n}{n+ 2 - \mu}}(\Omega)$ and every $u\in L^{2^*}(\Omega)$, there holds
\begin{equation}
\label{eq:moroz_vans_epsilon_interpolation}
    \int_\Omega I_\mu[P|u|^\theta]Q|u|^{2- \theta}
    \leq \epsilon\|u\|_{L^{2^*}(\Omega)}^2 + C_\epsilon\|u\|_{L^2(\Omega)}^2. 
\end{equation}
\end{lemma}
\begin{proof}
It suffices to prove the asserted inequality under the additional assumption that each of $P$, $Q$ and $u$ is nonnegative on $\Omega$. For $k\in \bb N$ and for $\varphi:\Omega\to[0, \infty)$ define 
\begin{equation*}
    \sigma_k(\varphi) = (\varphi - k)\chi_{\varphi^{-1}(k, \infty)}, 
\end{equation*}
so that with $\tau_k$ as in \eqref{eq:level_m_truncation} we have $\varphi = \tau_k(\varphi) + \sigma_k(\varphi)$ pointwise on $\Omega$. For each $\varphi\in \{P, Q\}$, as $k\to \infty$ we have $\tau_k(\varphi)\to \varphi$ a.e.\ in $\Omega$, $\tau_k(\varphi)\to \varphi$ in $L^{\frac{2n}{n + 2 - \mu}}(\Omega)$, $\sigma_k(\varphi)\to 0$ in $L^{\frac{2n}{n + 2 - \mu}}(\Omega)$, and $\tau_k(\varphi)\in L^\infty(\Omega)\subset L^{\frac{2n}{n- \mu}}(\Omega)$ (but we do not have $\tau_k(\varphi)\to \varphi$ in $L^{\frac{2n}{n- \mu}}(\Omega)$ for either of $\varphi = P$ or $\varphi = Q$). For each $k\in \bb N$ let us decompose 
\begin{equation}
\label{eq:the_four_Js}
    \int_\Omega I_\mu[Pu^\theta]Qu^{2- \theta}
    = \sum_{\ell = 1}^4 J_\ell(k), 
\end{equation}
where
\begin{equation*}
\begin{split}
    J_1(k) & = \int_\Omega I_\mu[\sigma_k(P)u^\theta]\sigma_k(Q)u^{2 - \theta}\\
    J_2(k) & = \int_\Omega I_\mu[\sigma_k(P)u^\theta]\tau_k(Q)u^{2 - \theta}\\
    J_3(k) & = \int_\Omega I_\mu[\tau_k(P) u^\theta]\sigma_k(Q)u^{2 - \theta}\\
    J_4(k) & = \int_\Omega I_\mu[\tau_k(P) u^\theta]\tau_k(Q) u^{2 - \theta}. 
\end{split}
\end{equation*}
For fixed $k$ we proceed to separately estimate each of $J_1(k), \ldots, J_4(k)$. To ease the notation we do not indicate the $k$-dependence in the notation for $J_1, \ldots, J_4$. We start with the estimates of $J_1$ and $J_4$ (since those two estimates are simpler). To estimate $J_1$ define $s_1$ and $t_1$ by 
\begin{equation}
\label{eq:s1t1}
	\frac 1{s_1} = \frac{n + 2 - \mu}{2n} + \frac{\theta}{2^*}
	\quad\text{ and }\quad
	\frac 1{t_1} = \frac{n + 2 - \mu}{2n} + \frac{2 - \theta}{2^*}. 
\end{equation}
The assumption $\theta\in (1 - \frac \mu n, 1 + \frac\mu n)$ implies that $\frac{1}{s_1}, \frac{1}{t_1}\in (0, 1)$ and a direct computation shows that $\frac{1}{s_1} + \frac{1}{t_1} + \frac\mu n = 2$, so the HLS inequality implies 
\begin{equation}
\label{eq:J1_initial_HLS}
    J_1
    \leq \mc H(n, \mu, s_1)\|\sigma_k(P)u^\theta\|_{s_1}\|\sigma_k(Q)u^{2- \theta}\|_{t_1}. 
\end{equation}
Moreover, from the definition of $s_1$ given in \eqref{eq:s1t1} and H\"older's inequality we have
\begin{equation*}
    \|\sigma_k(P)u^\theta\|_{s_1}
    \leq \|\sigma_k(P)\|_{\frac{2n}{n + 2 - \mu}}\|u\|_{2^*}^\theta. 
\end{equation*}
Similarly, using the definition of $t_1$ given in \eqref{eq:s1t1} we have
\begin{equation*}
    \|\sigma_k(Q)u^{2 - \theta}\|_{t_1}
    \leq \|\sigma_k(Q)\|_{\frac{2n}{n + 2 - \mu}}\|u\|_{2^*}^{2 - \theta}. 
\end{equation*}
Bringing the previous two estimates back to \eqref{eq:J1_initial_HLS} gives
\begin{equation}
\label{eq:J1_final}
    J_1
    \leq \mc H(n, \mu, s_1)\|\sigma_k(P)\|_{\frac{2n}{n + 2 - \mu}}\|\sigma_k(Q)\|_{\frac{2n}{n + 2 - \mu}}\|u\|_{2^*}^2. 
\end{equation}
Next we estimate $J_4$. Define $s_4$ and $t_4$ by 
\begin{equation}
\label{eq:s4t4}
	\frac 1{s_4} = \frac{n - \mu}{2n} + \frac{\theta}{2}
	\quad\text{ and }\quad
	\frac 1{t_4} = \frac{n - \mu}{2n} + \frac{2 - \theta}{2}. 
\end{equation}
The assumption $\theta\in (1 - \frac\mu n, 1 + \frac \mu n)$ guarantees that $\frac{1}{s_4}, \frac{1}{t_4}\in (0, 1)$ and a direct computation shows that $\frac{1}{s_4}+ \frac{1}{t_4} + \frac\mu n = 2$ so the HLS inequality implies
\begin{equation}
\label{eq:J4_initial_HLS}
    J_4
    \leq\mc H(n, \mu, s_4)\|\tau_k(P)u^\theta\|_{s_4}\|\tau_k(Q)u^{2- \theta}\|_{t_4}. 
\end{equation}
From the definition of $s_4$ given in \eqref{eq:s4t4} and H\"older's inequality we have
\begin{equation*}
    \|\tau_k(P)u^\theta\|_{s_4}
    \leq \|\tau_k(P)\|_{\frac{2n}{n - \mu}}\|u\|_2^\theta. 
\end{equation*}
Similarly, using the definition of $t_4$ given in \eqref{eq:s4t4} and H\"older's inequality we have
\begin{equation*}
    \|\tau_k(Q)u^{2 - \theta}\|_{t_4}
    \leq \|\tau_k(Q)\|_{\frac{2n}{n - \mu}}\|u\|_2^{2 - \theta}. 
\end{equation*}
Bringing the previous two estimates back to \eqref{eq:J4_initial_HLS} gives
\begin{equation}
\label{eq:J4_final}
    J_4
    \leq\mc H(n, \mu, s_4)\|\tau_k(P)\|_{\frac{2n}{n - \mu}}\|\tau_k(Q)\|_{\frac{2n}{n - \mu}}\|u\|_2^2. 
\end{equation}
Next we estimate $J_2$. To do so observe that the assumption $\theta\in (1 - \frac\mu n, 1 + \frac\mu n)$ guarantees that both 
\begin{equation*}
    \frac{n - 2 - \mu - (n -2)\theta}{2}< \min\{\theta, 1\}
\end{equation*}
and
\begin{equation*}
    \max\{0, \theta - 1\}
    \leq \frac{n - 2 + \mu - (n -2)\theta}{2}
\end{equation*}
from which we deduce the existence of 
\begin{equation}
\label{eq:beta2_first_interval}
    \beta_2\in (\max\{0, \theta - 1\}, \min\{\theta, 1\})
\end{equation}
for which 
\begin{equation}
\label{eq:beta2_theta_interval}
    n - 2 - \mu< (n - 2)\theta + 2\beta_2< n - 2 + \mu. 
\end{equation}
Fix any such $\beta_2$ and define $s_2$ and $t_2$ by 
\begin{equation}
\label{eq:s2}
    \frac 1{s_2}
    = \frac{n +2 - \mu}{2n} + \frac{\theta - \beta_2}{2^*} + \frac{\beta_2}2
\end{equation}
and
\begin{equation}
\label{eq:t2}
    \frac 1{t_2}
    = \frac{n - \mu}{2n} + \frac{1 - \beta_2}{2} + \frac{\beta_2- (\theta - 1)}{2^*}. 
\end{equation}
Using \eqref{eq:beta2_first_interval} and \eqref{eq:beta2_theta_interval} one can verify that $\frac{1}{s_2}, \frac{1}{t_2}\in (0, 1)$ and a direct computation shows that $\frac{1}{s_2} + \frac{1}{t_2} + \frac\mu n = 2$, so the HLS inequality implies 
\begin{equation}
\label{eq:J2_initial_HLS}
    J_2
    \leq \mc H(n, \mu, s_2)\|\sigma_k(P)u^\theta\|_{s_2}\|\tau_k(Q)u^{2- \theta}\|_{t_2}. 
\end{equation}
From \eqref{eq:beta2_first_interval}, \eqref{eq:s2}, and H\"older's inequality we have
\begin{equation*}
    \|\sigma_k(P)u^\theta\|_{s_2}
    \leq \|\sigma_k(P)\|_{\frac{2n}{n + 2- \mu}}\|u\|_{2^*}^{\theta - \beta_2}\|u\|_2^{\beta_2}. 
\end{equation*}
Similarly, from \eqref{eq:beta2_first_interval}, \eqref{eq:t2}, and H\"older's inequality we have
\begin{equation*}
    \|\tau_k(Q)u^{2- \theta}\|_{t_2}
    \leq \|\tau_k(Q)\|_{\frac{2n}{n - \mu}}\|u\|_2^{1 - \beta_2}\|u\|_{2^*}^{\beta_2 - (\theta - 1)}. 
\end{equation*}
Bringing the previous two estimates back to \eqref{eq:J2_initial_HLS} gives
\begin{equation}
\label{eq:J2_final}
    J_2
    \leq \mc H(n, \mu, s_2)\|\sigma_k(P)\|_{\frac{2n}{n + 2 - \mu}}\|\tau_k(Q)\|_{\frac{2n}{n - \mu}}\|u\|_2\|u\|_{2^*}. 
\end{equation}
Finally, we estimate $J_3$. To do so, note that the assumption $\theta\in (1 - \frac\mu n, 1 + \frac\mu n)$ guarantees that both 
\begin{equation*}
    \frac{n - \mu - (n - 2)\theta}{2}< \min\{\theta, 1\}
\end{equation*}
and
\begin{equation*}
    \max\{0, \theta - 1\}
    < \frac{n + \mu - (n - 2)\theta}{2}, 
\end{equation*}
from which we deduce the existence of 
\begin{equation}
\label{eq:beta3_first_interval}
    \beta_3\in (\max\{0, \theta - 1\}, \min\{\theta, 1\})
\end{equation}
such that 
\begin{equation}
\label{eq:beta3_theta_interval}
    n - \mu< (n - 2)\theta + 2\beta_3< n +\mu. 
\end{equation}
Fix any such $\beta_3$ and define $s_3$ and $t_3$ by 
\begin{equation}
\label{eq:s3}
    \frac1{s_3}
    = \frac{n - \mu}{2n} + \frac{\theta - \beta_3}{2^*} + \frac{\beta_3}{2}
\end{equation}
and
\begin{equation}
\label{eq:t3}
    \frac1{t_3}
    = \frac{n + 2 - \mu}{2n} + \frac{1 - \beta_3}{2} + \frac{\beta_3 - (\theta - 1)}{2^*}. 
\end{equation}
Using \eqref{eq:beta3_first_interval} and \eqref{eq:beta3_theta_interval} one can verify that $\frac{1}{s_3}, \frac{1}{t_3}\in (0, 1)$ and a direct computation shows that $\frac{1}{s_3} + \frac{1}{t_3} + \frac\mu n = 2$, so the HLS inequality implies
\begin{equation}
\label{eq:J3_initial_HLS}
    J_3
    \leq \mc H(n, \mu, s_3)\|\tau_k(P)u^\theta\|_{s_3}\|\sigma_k(Q)u^{2- \theta}\|_{t_3}. 
\end{equation}
From \eqref{eq:beta3_first_interval}, \eqref{eq:s3} and H\"older's inequality we have
\begin{equation*}
    \|\tau_k(P)u^\theta\|_{s_3}
    \leq\|\tau_k(P)\|_{\frac{2n}{n - \mu}}\|u\|_{2^*}^{\theta - \beta_3}\|u\|_2^{\beta_3}. 
\end{equation*}
Similarly, from \eqref{eq:beta3_first_interval}, \eqref{eq:t3} and H\"older's inequality we have
\begin{equation*}
    \|\sigma_k(Q)u^{2-  \theta}\|_{t_3}
    \leq \|\sigma_k(Q)\|_{\frac{2n}{n + 2 - \mu}}\|u\|_2^{1 - \beta_3}\|u\|_{2^*}^{\beta_3 - (\theta - 1)}. 
\end{equation*}
Bringing the previous two estimates back to \eqref{eq:J3_initial_HLS} gives
\begin{equation}
\label{eq:J3_final}
    J_3
    \leq \mc H(n, \mu, s_3)\|\tau_k(P)\|_{\frac{2n}{n - \mu}}\|\sigma_k(Q)\|_{\frac{2n}{n+ 2 - \mu}}\|u\|_2\|u\|_{2^*}. 
\end{equation}
Finally, bringing \eqref{eq:J1_final}, \eqref{eq:J2_final}, \eqref{eq:J3_final}, and \eqref{eq:J4_final} back to \eqref{eq:the_four_Js} then performing elementary estimates we find that if $C> \max\{\mc H(n, \mu, s_\ell): \ell = 1,\ldots , 4\}$ then 
\begin{equation*}
\begin{split}
    \int_\Omega & I_\mu[Pu^\theta]Qu^{2 - \theta}\\
    \leq &\; C\left(\|\sigma_k(P)\|_{\frac{2n}{n + 2 - \mu}} + \|\sigma_k(Q)\|_{\frac{2n}{n + 2 - \mu}}\right)^2\|u\|_{2^*}^2 \\
    & + C\left(\|\tau_k(P)\|_{\frac{2n}{n - \mu}} + \|\tau_k(Q)\|_{\frac{2n}{n - \mu}}\right)^2\|u\|_2^2.  
\end{split}
\end{equation*}
In view of this estimate and the fact that 
\begin{equation*}
	\lim_k\left(\|\sigma_k(P)\|_{\frac{2n}{n+ 2 - \mu}} + \|\sigma_k(Q)\|_{\frac{2n}{n + 2 - \mu}}\right) = 0, 
\end{equation*}
if $\epsilon>0$ is given we may choose $k$ sufficiently large so that inequality \eqref{eq:moroz_vans_epsilon_interpolation} holds. 
\end{proof}
\begin{lemma}
\label{lemma:general_integrability_boost}
Let $n\geq 3$, let $\Omega \subset \bb R^n$ be a bounded open set and let $\mu\in (0, n)$. Suppose $f\in C^0(\bb R^n; \bb R^d)$ is homogeneous of degree $1$ and let $P, Q\in L^{\frac{2n}{n+ 2 - \mu}}(\Omega; \bb R^d)$. If $U\in H_0^1(\Omega; \bb R^d)$ is a weak solution to 
\begin{equation}
\label{eq:linear_type_PDE}
    -\lap U = f(U) + I_\mu[P\cdot U]Q
    \quad \text{ in }\Omega 
\end{equation}
then for every $p\in [1, \frac{n}{n - \mu}\cdot \frac{2n}{n - 2})$ there holds $U\in L^p(\Omega; \bb R^d)$.
\end{lemma}
In view of Lemma \ref{lemma:properties_homogeneous_functions} \ref{item:homogeneous_grad_dot}, the improved integrability of weak solutions to problem \eqref{eq:main_problem} follows by applying Lemma \ref{lemma:general_integrability_boost} with $P=Q = h(U)$. 
\begin{coro}
\label{coro:integrability_boost}
Let $n\geq 3$, let $\Omega\subset\bb R^n$ be a bounded open set, and let $\mu\in(0, n)$. Suppose $F$ and $H$ satisfy \ref{item:F_homogeneous} and \ref{item:H_positively_homogeneous} respectively and let $f$ and $h$ be as in \eqref{eq:the_gradients}. If $U\in H_0^1(\Omega; \bb R^d)$ is a weak solution to problem \eqref{eq:main_problem} then for every $p\in [1, \frac{n}{n - \mu}\cdot \frac{2n}{n - 2})$ there holds $U\in L^p(\Omega; \bb R^d)$. 
\end{coro}
\begin{proof}[Proof of Lemma \ref{lemma:general_integrability_boost}]
Applying Lemma \ref{lemma:moroz_van_interpolation} with $\theta = 1$, Minkowski's inequality and the Sobolev inequality guarantees the existence of a constant $C_0>1$ such that the following inequality holds for all $V= (v_1, \ldots, v_d)\in H_0^1(\Omega; \bb R^d)$: 
\begin{equation}
\label{eq:implies_coercive}
    \int_\Omega I_\mu[|P||V|]|Q||V|
    \leq \frac 12\sum_{i = 1}^d\int_\Omega\left(|\Grad v_i|^2 + C_0v_i^2\right). 
\end{equation}
For $k\in \bb N$ define $P^k, Q^k: \Omega\to \bb R^d$ by truncating each coordinate function at level $k$: 
\begin{equation*}
\begin{split}
    P^k & = (\tau_k(P_1), \ldots, \tau_k(P_d))\\
    Q^k & = (\tau_k(Q_1), \ldots, \tau_k(Q_d)), 
\end{split}
\end{equation*}
where $\tau_k$ is as in \eqref{eq:level_m_truncation}. Define the bilinear form $a_k$ on $H_0^1(\Omega; \bb R^d)$ by 
\begin{equation*}
    a_k(U, V)
    = \sum_{i = 1}^d\int_\Omega(\Grad u_i\cdot \Grad v_i + C_0u_iv_i)
    - \int_\Omega I_\mu[P^k\cdot U]Q^k\cdot V, 
\end{equation*}
where $C_0$ is as in \eqref{eq:implies_coercive}. Evidently $a_k$ is bounded in the sense of bilinear forms and the choice of $C_0$ together with \eqref{eq:implies_coercive} and the pointwise inequalities $|P^k|\leq |P|$ and $|Q^k|\leq |Q|$ guarantee the coercivity of $a_k$. 
Now let $U$ be a solution to \eqref{eq:linear_type_PDE}. For each $k$ the Lax-Milgram Theorem ensures the existence of a unique $U^{(k)}= (u^{(k)}_1, \ldots, u^{(k)}_d)\in H_0^1(\Omega; \bb R^d)$ such that 
\begin{equation}
\label{eq:Uk_weak}
    a_k(U^{(k)}, V)
    = \int_\Omega (f(U)+ C_0U)\cdot V
    \quad \text{ for all }V\in H_0^1(\Omega; \bb R^d). 
\end{equation}
In particular, for every $i\in \{1, \ldots, d\}$ $u_i^{(k)}$ satisfies
\begin{equation}
\label{eq:uk_coordinate_equations}
    -\lap u^{(k)}_i + C_0u^{(k)}_i
    = I_\mu[P^k\cdot U^{(k)}]Q^k_i + f_i(U) + C_0 u_i
    \quad \text{ in }\Omega 
\end{equation}
in the weak sense. A standard argument using the fact that $U^{(k)}$ satisfies \eqref{eq:Uk_weak} and the fact that $U$ satisfies \eqref{eq:linear_type_PDE} shows that 
$U^{(k)}\weakconv U$ weakly in $H_0^1(\Omega; \bb R^d)$. 
For any $p\geq 2$, and any $m\in \bb N$ the functions $|\tau_m(u_i^{(k)})|^{p - 2}\tau_m(u_i^{(k)})$ are valid test functions for the weak formulation of problem \eqref{eq:uk_coordinate_equations}. Therefore, 
\begin{equation}
\label{eq:initial_testing}
\begin{split}
    \sum_{i = 1}^d\int_\Omega & \Grad u^{(k)}_i\cdot \Grad\left(|\tau_m(u_i^{(k)})|^{p - 2}\tau_m(u_i^{(k)})\right)\\
    = & \; \sum_{i = 1}^d \int_\Omega I_\mu[P^k\cdot U^{(k)}]Q^k_i |\tau_m(u_i^{(k)})|^{p - 2}\tau_m(u_i^{(k)})\\ 
    & + \sum_{i = 1}^d\int_\Omega \left(f_i(U) + C_0(u_i - u^{(k)}_i)\right)|\tau_m(u_i^{(k)})|^{p - 2}\tau_m(u_i^{(k)}). 
\end{split}
\end{equation}
A direct computation using 
%
the fact that $\Grad u^{(k)}_i= \Grad (\tau_m(u^{(k)}_i))$ on $\supp\Grad (\tau_m(u^{(k)}_i))$ 
 shows that the $i^{\text{th}}$ integrand on the left-hand side of \eqref{eq:initial_testing} satisfies
\begin{equation*}
    \frac{4(p - 1)}{p^2}|\Grad(|\tau_m(u^{(k)}_i)|^{p/2})|^2
    = \Grad u^{(k)}_i\cdot \Grad\left(|\tau_m(u^{(k)}_i)|^{p - 2}\tau_m(u^{(k)}_i)\right)
\end{equation*}
a.e.\ in $\Omega$. 
Using this equality, the Sobolev inequality, Minkowski's inequality, and Lemma \ref{lemma:vector_truncation} we bound the left-hand side of \eqref{eq:initial_testing} from below as follows: 
\begin{equation}
\label{eq:sobolev_side}
\begin{split}
    \frac{p^2\mc S}{4(p - 1)}\sum_{i = 1}^d\int_\Omega & \Grad u^{(k)}_i\cdot \Grad\left(|\tau_m(u_i^{(k)})|^{p - 2}\tau_m(u_i^{(k)})\right)\\
    = & \; \mc S\sum_{i = 1}^d\int_\Omega |\Grad(|\tau_m(u^{(k)}_i)|^{p/2})|^2\\
    \geq & \; \sum_{i = 1}^d\left(\int_\Omega (|\tau_m(u^{(k)}_i)|^p)^{\frac{n}{n - 2}}\right)^{\frac{n - 2}n}\\
    \geq & \; \left(\int_\Omega (\sum_{i = 1}^d|\tau_m(u^{(k)}_i)|^p)^{\frac{n}{n - 2}}\right)^{\frac{n - 2}n}\\
    \geq & \; \frac{1}{C_1(d, p)}\left(\int_\Omega \tau_m(|U^{(k)}|)^{\frac{np}{n - 2}}\right)^{\frac{n - 2}n}. 
\end{split}
\end{equation}
Next we consider the second term on the right-hand side of \eqref{eq:initial_testing}. Writing $f = (f_1, \ldots, f_d)$, the homogeneity assumption of $f$ guarantees that each $f_i$ is homogeneous of degree one. Moreover, item \ref{item:homogeneous_upper_bound} of Lemma \ref{lemma:properties_homogeneous_functions} guarantees that for each $i\in\{1, \ldots, d\}$, there holds $|f_i(s)|\leq M|s|$ for all $s\in \bb R^d$, where $M = \max\{|f(s)|: s\in \bb S^{d - 1}\}$. The $i^{\text{th}}$ summand of the second term on the right-hand side of \eqref{eq:initial_testing} is bounded above using Young's inequality 
as follows: 
\begin{equation*}
\begin{split}
    \int_\Omega & \left(f_i(U) + C_0(u_i - u_i^{(k)})\right)|\tau_m(u_i^{(k)})|^{p - 2}\tau_m(u_i^{(k)})\\
    & \leq \int_\Omega\left((M + C_0)|U| + C_0|U^{(k)}|\right)|\tau_m(u_i^{(k)})|^{p - 1}\\
    & \leq \int_\Omega \left(\frac{M + C_0}p|U|^p + \frac{C_0}p|U^{(k)}|^p + \frac{2}{p'}|\tau_m(u_i^{(k)})|^p\right).
\end{split}
\end{equation*}
Summing this estimate over $i$ and using Lemma \ref{lemma:vector_truncation} gives
\begin{equation}
\label{eq:RHS_tested_minor_terms}
\begin{split}
   \sum_{i = 1}^d \int_\Omega & \left(f_i(U) + C_0(u_i - u_i^{(k)})\right)|\tau_m(u_i^{(k)})|^{p - 2}\tau_m(u_i^{(k)})\\
   & 
    \leq C(M, C_0, p, d)\int_\Omega \left(|U|^p + |U^{(k)}|^p + \tau_m(|U^{(k)}|)^p\right).
\end{split}
\end{equation}
Next we consider the first term on the right-hand side of \eqref{eq:initial_testing}. Using Lemma \ref{lemma:vector_truncation} and the symmetry of $I_\mu$ we have
\begin{equation}
\label{eq:RHS_tested_major_term}
\begin{split}
   \sum_{i = 1}^d \int_\Omega & I_\mu[P^k\cdot U^{(k)}]Q_i^k|\tau_m(u^{(k)}_i)|^{p - 2}\tau_m(u^{(k)}_i)\\
   \leq & \int_\Omega I_\mu[|P^k| |U^{(k)}|]|Q^k|\sum_{i = 1}^d |\tau_m(u^{(k)}_i)|^{p - 1}\\
   \leq &\;  C_1(d, p - 1)\int_\Omega I_\mu[|P^k| |U^{(k)}|]|Q^k|\tau_m(|U^{(k)}|)^{p - 1}\\
   \leq & \; C_1\int_\Omega I_\mu[|P^k|\tau_m(|U^{(k)}|)]|Q^k|\tau_m(|U^{(k)}|)^{p - 1}\\
   & + C_1\int_{\Omega\cap\{|U^{(k)}|> m\}}I_\mu[|Q^k| |U^{(k)}|^{p - 1}]|P^k||U^{(k)}|. 
\end{split}
\end{equation}
Suppose $p\in [2, \frac{2n}{n - \mu})$ is chosen so that $|U^{(k)}|\in L^p(\Omega)$. Since $|P^k|, |Q^k|\in L^{\frac{2n}{n - \mu}}(\Omega)$ we have $I_\mu[|P^k||U^{(k)}|]|Q^k||U^{(k)}|^{p - 1}\in L^1(\Omega)$ and thus, the last term on the right-most side of \eqref{eq:RHS_tested_major_term} satisfies
\begin{equation}
\label{eq:minor_m}
    \lim_{m\to\infty}\int_{\Omega\cap\{|U^{(k)}|> m\}}I_\mu[|Q^k| |U^{(k)}|^{p - 1}]|P^k||U^{(k)}|
    = 0.
\end{equation}
Moreover, applying Lemma \ref{lemma:moroz_van_interpolation} with $\theta = \frac 2p$ 
we find that for any $\epsilon>0$ there is $C_\epsilon>0$ such that 
\begin{equation}
\label{eq:moroz_vans_epsilon_interpolation_applied}
\begin{split}
    \int_\Omega & I_\mu[|P^k|\tau_m(|U^{(k)}|)]|Q^k|\tau_m(|U^{(k)}|)^{p - 1}\\
    & 
    \leq  \epsilon\left(\int_\Omega \tau_m(|U^{(k)}|)^{\frac{np}{n - 2}}\right)^{\frac{n- 2}n} + C_\epsilon\int_\Omega \tau_m(|U^{(k)}|)^p.
\end{split}
\end{equation}
For $\epsilon>0$ sufficiently small, bringing \eqref{eq:minor_m} and \eqref{eq:moroz_vans_epsilon_interpolation_applied} back to \eqref{eq:RHS_tested_major_term}, then bringing the resultant estimate together with \eqref{eq:sobolev_side} and \eqref{eq:RHS_tested_minor_terms} back to \eqref{eq:initial_testing} we have
\begin{equation*}
\begin{split}
    \left(\int_\Omega \tau_m(|U^{(k)}|)^{\frac{np}{n - 2}}\right)^{\frac{n - 2}n}
    \leq & \; C\int_\Omega\left(|U|^p + |U^{(k)}|^p + \tau_m(|U^{(k)}|)^p\right) +\circ(1), 
\end{split}
\end{equation*}
where $\circ(1)\to 0$ as $m\to \infty$ and $C$ is independent of both $m$ and $k$. Since $|U^{(k)}|\in L^p$ 
 Fatou's Lemma gives
\begin{equation*}
\begin{split}
    \left(\int_\Omega |U^{(k)}|^{\frac{np}{n - 2}}\right)^{\frac{n - 2}n}
    \leq & \; \liminf_m\left(\int_\Omega \tau_m(|U^{(k)}|)^{\frac{np}{n - 2}}\right)^{\frac{n - 2}n}\\
    \leq & \; C\int_\Omega\left(|U|^p + |U^{(k)}|^p\right). 
\end{split}
\end{equation*}
Now we allow $k$ to vary. The previous estimate implies 
\begin{equation}
\label{eq:limsup_k_estimate}
\begin{split}
    \limsup_k\left(\int_\Omega |U^{(k)}|^{\frac{np}{n - 2}}\right)^{\frac{n - 2}n}
    \leq & \; C\limsup_k\int_\Omega\left(|U|^p + |U^{(k)}|^p\right). 
\end{split}
\end{equation}
Starting with $p = p_0\in [2, \min\{2^*, \frac{2n}{n - \mu}\})$, since $U^{(k)}\weakconv U$ weakly in $H_0^1(\Omega; \bb R^d)$, $(U^{(k)})_{k = 1}^\infty$ is bounded in $H_0^1(\Omega; \bb R^d)$, so $(U^{(k)})_{k = 1}^\infty$ is also bounded in $L^{p_0}(\Omega; \bb R^d)$. Estimate \eqref{eq:limsup_k_estimate} implies that $(U^{(k)})_{k =1}^\infty$ is bounded in $L^{p_1}(\Omega; \bb R^d)$ where $p_1 = np_0/(n - 2)$. Compactness of the embedding $H_0^1(\Omega; \bb R^d)\hookrightarrow L^{p_0}(\Omega; \bb R^d)$ implies the existence of a subsequence of $(U^{(k)})_{k = 1}^\infty$ that converges to $U$ in $L^{p_0}(\Omega; \bb R^d)$. Passing to a further subsequence we may assume in addition that $U^{(k)}\to U$ a.e.\ in $\Omega$. Therefore, Fatou's Lemma gives
\begin{equation*}
\begin{split}
    \left(\int_\Omega|U|^{\frac{np_0}{n- 2}}\right)^{\frac{n - 2}n}
    \leq & \; \liminf_k\left(\int_\Omega|U^{(k)}|^{\frac{np_0}{n- 2}}\right)^{\frac{n - 2}n}\\
    \leq & \; \limsup_k\left(\int_\Omega|U^{(k)}|^{\frac{np_0}{n- 2}}\right)^{\frac{n - 2}n}\\
    \leq & \; C\int_\Omega |U|^{p_0}, 
\end{split}
\end{equation*}
and thus $U\in L^{p_1}(\Omega; \bb R^d)$. Using this estimate and by separately considering the case $\mu\in (0, 2]$ and the case $\mu\in (2, n)$, it is routine to verify that $U\in L^p(\Omega; \bb R^d)$ for all $p\in [1, \frac{2n}{n - 2}\cdot \frac{n}{n- \mu})$. 
\end{proof}
\begin{proof}[Proof of Proposition \ref{prop:C1alpha_regularity}]
Lemma \ref{lemma:homogeneous_integrability} and Corollary \ref{coro:integrability_boost} guarantee the existence of $q> n/(n - \mu)$ for which $H(U)\in L^q(\Omega)$, from which we deduce that $I_\mu[H(U)]\in L^\infty(\Omega)$. Setting $a = I_\mu[H(U)]\in L^\infty(\Omega)$, for every $i\in \{1, \ldots, d\}$, we have
\begin{equation}
\label{eq:bounded_coefficient_problem}
    -\lap u_i= f_i(U) + a(x)h_i(U).  
\end{equation}
Moreover, for every $p\in [1, \frac{2n}{n -2}\cdot \frac{n}{n - \mu})$ we have the containments $u_i\in L^p(\Omega) $, $f_i\circ U\in L^p(\Omega)$, and $h_i\circ U\in L^{p/(2^*_\mu - 1)}(\Omega)$. The assertion of Proposition \ref{prop:C1alpha_regularity} is established via a standard bootstrap-style argument. For the reader's convenience we outline the details. In what follows we separately consider the case $\mu\in [4, n)$ and the case $\mu\in(0, 4)$. 
\begin{enumerate}[label = {\bf Case \arabic*.}, ref = {\bf Case \arabic*}, wide = 0pt]
    \item \label{item:mu_large_regularity} $\mu\in [4, n)$.\\
    In this case we have $1< 2^*_\mu\leq 2$ so the \ref{item:mu_large_regularity} assumption and the fact that $u_i$ satisfies \eqref{eq:bounded_coefficient_problem} guarantees that $-\lap u_i\in L^p(\Omega)$ for all $p\in [1, 2^*\cdot\frac{n}{n - \mu})$. Since $\frac{2n}{n + 2}< 2^*\cdot\frac{n}{n - \mu}$, the standard bootstrap argument gives $U\in C^{1, \alpha}(\bar\Omega; \bb R^d)$ for some $\alpha\in (0, 1)$. 
    \item \label{item:mu_small_regularity} $\mu\in (0, 4)$.\\
    In this case we have $2^*_\mu - 1 > 1$ and thus for every $i\in \{1, \ldots, d\}$ and every $p\in [1, \frac{2n}{n - 2}\cdot \frac{n}{n - \mu})$ we have $-\lap u_i\in L^{p/(2^*_\mu - 1)}(\Omega)$. In what follows we separately consider the subcase $(n - \mu)(n + 2 - \mu)< 4n$ and the subcase $(n - \mu)(n + 2 - \mu)\geq 4n$. 
    \begin{enumerate}[label = {\bf Subcase  2(\alph*).}, ref = {\bf Subcase  2(\alph*)}, wide = 0pt]
        \item $(n - \mu)(n + 2 - \mu)< 4n$. \\
        In this case there is $p_0\in (\frac n2, \frac{2n}{n - 2}\cdot \frac{n}{n - \mu}\cdot \frac 1{2^*_\mu - 1})$ such that for all $i\in \{1, \ldots, d\}$ there holds $-\lap u_i\in L^{p_0}(\Omega)$. Standard arguments using the elliptic $L^p$ theory, the Sobolev embedding and the assumption that $\bdy\Omega\in C^{1, 1}$ guarantee the existence of $\alpha\in (0, 1)$ such that $u_i\in C^{1, \alpha}(\bar \Omega)$ for all $i$. 
        \item $(n - \mu)(n + 2 - \mu)\geq 4n$.\\
        In this case choose $p_0\in (\frac n2\cdot \frac{4 - \mu}{n + 2 - \mu}, \frac{2n^2}{(n - \mu)(n + 2 - \mu)})$ and define $N = N(p_0)$ by 
        \begin{equation}
        \label{eq:num_iterations}
            N 
            = \min\left\{\ell\in \bb N\cup\{0\}: \frac{1}{p_0}< \frac 2n\sum_{k = 0}^{\ell + 1}(2^*_\mu - 1)^{-k}\right\}. 
        \end{equation}
        The choice of $p_0$ ensures that $N$ is well-defined and that $N\in \bb N\cup \{0\}$. 
    Defining $p_1, \ldots, p_{N + 1}$ by 
    \begin{equation}
    \label{eq:pell}
        \frac 1{p_\ell} = \frac{(2^*_\mu - 1)^\ell}{p_0} - \frac 2n\sum_{k = 1}^\ell (2^*_\mu - 1)^k;
        \quad \text{ for }\ell = 1, \ldots, N + 1
    \end{equation}
    we have
    \begin{equation*}
        p_{\ell + 1}
        = \frac{1}{2^*_\mu - 1}\cdot \frac{np_\ell}{n - 2p_\ell}
        \quad \text{ for }\ell = 0, \ldots, N. 
    \end{equation*}
    Moreover, the definition of $N$ in \eqref{eq:num_iterations} guarantees that $p_N\leq \frac n2< p_{N + 1}$. 
    In fact, by decreasing $p_0$ slightly if necessary we may assume that $p_N< \frac n2< p_{N + 1}$. Now we iteratively improve the integrability of $-\lap u_i$. For each $i\in \{1, \ldots, d\}$, starting with $-\lap u_i\in L^{p_0}(\Omega)$, the elliptic $L^p$ theory and the Sobolev embedding guarantee that $u_i\in L^{\frac{np_0}{n - 2p_0}}(\Omega)$. 
    Since $u_i$ satisfies \eqref{eq:bounded_coefficient_problem} and in view of the \ref{item:mu_small_regularity} assumption, we find that $-\lap u_i\in L^{p_1}(\Omega)$. For $\ell = 1, \ldots, N$, having shown that $-\lap u_i\in L^{p_\ell}(\Omega)$, the same argument shows that $-\lap u_i\in L^{p_{\ell + 1}}(\Omega)$. In particular we have $-\lap u_i\in L^{p_{N + 1}}(\Omega)$ so since $p_{N + 1}> \frac n2$, the elliptic $L^p$ theory, the Sobolev embedding and the assumption that $\bdy \Omega\in C^{1, 1}$ give $u_i\in C^{1, \alpha}(\bar\Omega)$ for some $\alpha\in (0, 1)$. 
\end{enumerate}
\end{enumerate}
\end{proof}
%
%
%

\input{semilinear_elliptic_systems_nonlocal_homogeneous_critical_nonlinearities_arxiv.bbl}

\end{document}

%% file: structure.tex
\usepackage[T1]{fontenc} 
\usepackage[utf8]{inputenc} 
\usepackage{amssymb}
\usepackage{bbm} 
\usepackage{graphicx} 
\usepackage{subcaption} 
\usepackage[dvipsnames]{xcolor}
\graphicspath{{Figures/}} 
\usepackage{enumitem} 
\setlist[enumerate, 1]{label = {\bf\arabic*.}, ref = {\bf\arabic*}, leftmargin=*} 
\setlist[enumerate, 2]{label = {\bf(\alph*)}, leftmargin=*} 
\usepackage{varioref} 
\usepackage{tikz-cd} 
\usetikzlibrary{arrows,arrows.meta, calc}
\tikzcdset{every arrow/.append style = -{Latex[scale=1]}}
\usepackage{bm}
\usepackage{slashed} 
\usepackage{esint} 
\usepackage{amsaddr} 
\usepackage{blkarray}
\usepackage{nicematrix}
\usepackage{amsthm}
\theoremstyle{plain} 
\newtheorem{theorem}{Theorem}
\newtheorem{lemma}[theorem]{Lemma}
\newtheorem{prop}[theorem]{Proposition}
\newtheorem{coro}[theorem]{Corollary}

\newtheorem{oldtheorem}{Theorem}

\theoremstyle{definition} 
\newtheorem{defn}[theorem]{Definition}

\theoremstyle{remark} 
\newtheorem{remark}[theorem]{Remark}

\numberwithin{equation}{section}
\numberwithin{theorem}{section}

\usepackage[bookmarks, colorlinks=true, pdfstartview=FitV, linkcolor=blue, citecolor=blue, urlcolor=blue]{hyperref}

\renewcommand{\d}{\mathrm d}
\newcommand{\lap}{\Delta}
\newcommand{\Grad}{\nabla}
\newcommand{\bdy}{\partial}
\newcommand{\weakconv}{\rightharpoonup}

\DeclareMathOperator{\supp}{supp}

\newcommand{\bb}{\mathbb}
\newcommand{\mc}{\mathcal}
\newcommand{\abs}[1]{\left|#1\right|}

\newcommand{\lb}{\langle}
\newcommand{\rb}{\rangle}

%% file: semilinear_elliptic_systems_nonlocal_homogeneous_critical_nonlinearities_arxiv.bbl
\newcommand{\etalchar}[1]{$^{#1}$}